\documentclass[letterpaper,11pt]{article}
\pdfoutput=1 
\usepackage[utf8]{inputenc}
\usepackage{times,natbib,graphicx} 
\usepackage{doi} 
\usepackage{url}
\usepackage{hyperref}
\hypersetup{colorlinks = true,citecolor=blue,urlcolor=cyan}
\usepackage{amssymb,amsmath,amsthm,amscd,mathrsfs, cite,setspace}
\usepackage[margin=1.3in]{geometry}
\usepackage{caption}
\numberwithin{equation}{section}
\newtheorem{theorem}{Theorem}
\newtheorem{lemma}{Lemma}
\newtheorem{corollary}{Corollary}

\begin{document}

\title{Ulam's History-dependent Random Adding Process} 
\author
{{\sc \normalsize  By PETER CLIFFORD} \vspace{0.5ex} \\
{\em \normalsize Jesus College, Oxford, OX1 3DW, UK} \vspace{1.5ex} \\
{\sc \normalsize  and DAVID STIRZAKER} \vspace{0.5ex} \\
{\em \normalsize St.\ John's College, Oxford,  OX1 3JP, UK}}

\date{\vspace{-1ex}}

\maketitle
 
\begin{abstract}
Physical systems that have no memory have been very effectively modelled by Markov processes in their several forms.  But many real-world systems clearly do not have this property, including e.g.\ foraging animals, stochastic learning machines, and quantum dynamic processes.  There is therefore much interest in history-dependent processes whose evolution depends on all or part of their prior sample paths.  

Here we consider a large family of history-dependent growth processes, the simplest of which was first defined by Stanislaw Ulam in this recursion for a random sequence of positive integers: 
$X_{n+1} = X_{U(n)} + X_{V (n)}$, $n \geqslant r$ where $U(n)$ and $V (n)$ are independently and uniformly distributed on $\{1, \dots , n\}$, and the initial sequence, $X_1 = x_1, \dots , X_r = x_r$, is fixed. 
We consider the asymptotic properties of this sequence as $n \to \infty$, showing, for example, that $n^{-2} \sum_{k=1}^n X_k$ converges to a non-degenerate random variable.  We also consider the moments and auto-covariance of the process, showing, for example, that when the initial condition is $x_1 =1$ with $r =1$, then $\lim_{n\to \infty} n^{-2} E X^2_n = (2 \pi)^{-1} \sinh(\pi)$; and that for large $m < n$, we have $(m n)^{-1} E X_m X_n \doteq (3 \pi)^{-1} \sinh(\pi)$.  

We further consider new random adding processes where changes occur independently at discrete times with probability $p$, or where changes occur continuously at jump times of an independent Poisson process.  The processes are shown to have properties similar to those of the discrete time process with $p=1$, and to be readily generalised to a wider range of related sequences.\\


\end{abstract}

\section{Introduction}

Much of the theory of random processes has been driven by the study of evolving physical systems that clearly either have no memory, or can be assumed forgetful to a good approximation.  This leads naturally to models that require the Markov property in one of its forms.  However, many real-world systems exhibit marked long-range dependence, together with the phenomenon of `lock-in', (or `self-organization'), and other manifestations of non-ergodic behaviour.  The customary assumption of the Markov property, though extraordinarily convenient, limits the extent to which such properties can feature in the development of the process. A natural next step therefore is to abandon that restriction, and allow the short-term development of the process to have an explicit dependence on any part, or all, of its history.   Applications include:  history dependent (HD) quantum dynamics \citep{BL},
HD dynamic random utility \citep{FIS},
HD predator-prey models \citep{GE},
HD materials \citep{MBE},
HD social networks \citep{CBHT,PS}, 
and so on.

Models for describing such phenomena fall into several types, and for a broad survey see \citet{P}.
Among the various types of HD random processes, perhaps the most developed strand comprises the HD random walks (RW); these are known in general as self-exciting RW, or self-reinforcing RW, or self-interacting RW, or self-avoiding RW, and so on.  Further sub-divisions arise according as the walk may be edge-reinforced, vertex reinforced, step reinforced, etc.  See for example \citet{KM,B}.


Notable special cases include: first the Shark Random Swim \citep{BU}; second 
the Elephant Random Walk (ERW), first introduced by \citet{SGT}, 
and see also \citet{BE}; and third, the Reverting Random Walk \citep{BR,CS1}.

The field now extends to HD Brownian motion and HD L\'{e}vy processes \citep{BERT},
and HD spatio-temporal processes \citep{R}.

Note that in many applications, especially the social sciences, such processes are called path-dependent, but this term is also used in stochastic analysis to denote the solution of a stochastic differential equation whose coefficients depend on the paths of another process, such as the Wiener process.

Here, we consider a type of HD growth process, first suggested by Stanislaw Ulam where the next step depends on the entire past of the process.  Specifically, every member of the sequence of values is the sum of two values chosen from the previous history.

This type of sequence was studied in discrete time by \citet{BSU}, thus
\begin{equation}\label{eq:BSU}
X_{n+1} = X_{U(n)} + X_{V(n)}, \ \ n \geqslant 2,
\end{equation}
where $X_1=x_1$ and $X_2=x_2$ are given, and $(U(n),V(n); n \geqslant 1)$ comprise a sequence of independent random variables such that for given $n$, $U(n)$ and $V(n)$ are each uniformly distributed on $\{1,\dots,n\}$.  They note that $E X_n = \frac13 (x_1+x_2)n,$ for  $n\geqslant 3,$ and they conjectured from computer simulations that $E X^2_n$ grows quadratically with $n$ as $n \to \infty$. (They made 5000 simulations each with 100 steps.) They also note that since the process $(X_n; n \geqslant 1)$ does not enjoy the Markov property, or similar simplifications, it is not straightforward to analyse.

The sequence defined in \eqref{eq:BSU} was later considered by \citet{BNK}, with the initial condition $X_1=x_1.$ They note that in this case $E X_n = n x_1, n \geqslant 1.$  On the basis of further simulations ($10^8$ realisations, each of $1000$ steps), they conjectured that $E X^2_n$ grows quadratically and also that $E X^3_n$ grows with the cube of $n$.  

We shall verify these conjectures, and identify a martingale that further elucidates the behaviour of $(X_n; n \geqslant 1)$.  We will then consider a new randomised adding sequence, in which changes occur randomly with probability $p$.  Finally we consider a related adding process in continuous time, in which changes are regulated by a Poisson process.  Such processes have previously been introduced in the context of history-dependent growth processes \citep{CS}.  The process is shown to reproduce, in continuous time, the essential properties of Ulam's discrete time sequence.  Furthermore, similar analyses can be made of many more general processes, which we briefly outline.

\section{The adding process in discrete time}\label{sec:2}
Consider the process defined in \eqref{eq:BSU} with initial fixed sequence $(X_1=x_1,\dots,X_r=x_r)$ and let $s_r = \sum_{k=1}^r x_k$ and $t_r=\sum_{k=1}^r x_k^2$.   Denote the mean of $X_n$ by $m_n = E X_n$.  By conditional expectation,  
$$m_{n+1} = \frac{2}{n} \sum_{k=1}^n m_k, \quad n \geqslant r, $$
and an easy induction gives 
\begin{equation} \label{eq:martmean}
m_n = \frac{2 n}{r(r+1)} \sum_{k=1}^r x_k, \quad n > r.
\end{equation}
For the second moment we have this:

\begin{theorem}  \label{thm:qlimit}
\begin{equation} \label{eq:qlimit} 
\frac{E X^2_n}{n^2} \to  K(x_1,\dots,x_r),\quad \text{as $n \to \infty$},
\end{equation} 
 where
\begin{equation} \label{eq:K}
K(x_1,\dots,x_r) 
= \frac{\sinh(\pi)}{2\pi W_r} \left\{2 (r+1)\left(\frac{t_r}{r} + \frac{s_r^2}{r^2}\right) -(r+2)x_r^2\right\},
\end{equation}
and
\begin{equation}
W_r=\frac12 \left\{(r+1)^2+1\right\} \prod_{k=1}^r\left(1+\frac{1}{k^2}\right). 
\end{equation}
\end{theorem}  
\begin{proof}
Define $S_n = \sum_{k=1}^n X_k$ and let $p_n = E S_n^2$ and $q_n= E X^2_n$. By conditional expectation, for $n \geqslant r$,
\begin{equation} \label{eq:qn+1}
q_{n+1} = \frac{2}{n} \sum_{k=1}^n q_k + \frac{2}{n^2} p_n.
\end{equation}
Also by conditional expectation, 
\begin{equation} \label{eq:an+1}
p_{n+1} = E (S^2_n + 2 S_n X_{n+1} + X^2_{n+1}) = \frac{n+4}{n} p_n + q_{n+1}.
\end{equation}
Eliminating $p_n$, we have
\begin{equation} \label{eq:LRadding}
(n+1)^2 q_{n+2} - 2(n+1)(n+2)q_{n+1} + \{(n+2)^2+1\} q_n = 0,
\end{equation}
with initial conditions
\begin{equation} \label{eq:ICadding}
q_r = x_r^2 \quad  \text{and} \quad  q_{r+1} =  \frac{2t_r}{r} +  \frac{2s^2_r}{r^2}.
\end{equation}

By inspection, a particular solution of \eqref{eq:LRadding} is $q^*_n = n+1$.  From the theory of difference equations \citep{Elaydi}, a second, linearly independent, solution $q^\circ_n$ is given by 
$$
q^\circ_n = q^*_n \sum_{k=0}^{n-1} \frac{W_k}{(k+1)(k+2)},\quad n>r ,$$
where $W_k$ is the Casoratian associated with the difference equation.  We may set $W_0 =1$ and in this instance $W_n$ is given by the recursion
\begin{eqnarray*}
W_{n+1} &=& \frac{(n+2)^2+1}{(n+1)^2} W_n\\
       &=& \frac12 \left\{(n+2)^2+1\right\} \prod_{k=1}^{n+1} \left(1+ \frac{1}{k^2}\right)\\
       &\sim& n^2 (2 \pi)^{-1} \sinh(\pi),
\end{eqnarray*}
where we have used the product limit of \citet[\S 156]{Euler} and where the notation $f(n) \sim g(n)$ indicates that $\lim_{n\to \infty} f(n)/g(n) =1.$

The general solution of \eqref{eq:LRadding} is given by $q_n = A(n+1) + Bq^{\circ}_n$, where the constants $A$ and $B$ are determined by the initial conditions \eqref{eq:ICadding}.  Hence for $n> r$, 
$$q_n = \frac{(n+1)x_r^2}{r+1} + \frac{(n+1)}{W_r}\left\{2(r+1)\left(\frac{t_r}{r} + \frac{s_r^2}{r^2}\right) -(r+2)x_r^2)\right\}\sum_{k=r}^{n-1}\frac{W_k}{(k+1)(k+2)},$$ and since $W_n \sim n^2 (2\pi)^{-1} \sinh(\pi)$, the limit \eqref{eq:qlimit} follows.   

For the special case considered by \citet{BNK}, the constant $K$ becomes $\sinh(\pi)/( 2 \pi) \doteq 1.83804$ in good agreement with the approximate value of $1.84$ that they obtained by simulation.\end{proof}
Higher moments can be obtained in a similar fashion.  For simplicity we restrict attention to the special case with initial condition $x_1=1$, as in \citet{BNK}. For the third moment $t_n = E(X_n^3)$ we define 
$$
a_n^{[j,k]} = E\left\{X_n^j S^k_{n-1}\right\}  \quad \text{and} \quad  
b_n^{[j,k]} = E\left\{\left(\textstyle{\sum_{\nu=1}^n X^j_\nu}\right) S^k_n\right\}.
$$ 
By the usual conditional expectation arguments we find
\begin{eqnarray*}
a_{n+1}^{[0,3]} &=& a_{n}^{[0,3]} + 3 a_{n}^{[1,2]}  + 3 a_{n}^{[2,1]}  + a_{n}^{[3,0]},\quad 
a_{n+1}^{[1,2]}  =  \frac{2}{n} a_{n+1}^{[0,3]},\\
a_{n+1}^{[2,1]} &=& \frac{2}{n^2} a_{n+1}^{[0,3]} + \frac{2}{n} b_{n}^{[2,1]}, \quad
a_{n+1}^{[3,0]}  =  \frac{2}{n} b_{n}^{[3,0]} + \frac{6}{n^2} b_{n}^{[2,1]},\\
b_{n+1}^{[2,1]} &=& a_{n}^{[0,3]}\left(1+\frac{2}{n}\right) + a_{n+1}^{[2,1]} + a_{n+1}^{[3,0]}, \quad
b_{n+1}^{[3,0]} =  b_{n}^{[3,0]} + a_{n+1}^{[3,0]},
\end{eqnarray*}
with initial conditions $a_1^{[j,k]} = 0$ and $b_1^{[j,k]} = 1$. Reducing this system to a single recurrence for $t_n = a_{n}^{[3,0]}$ yields
\begin{equation}\label{eq:3rdmoment}
\begin{split}
(4 n-3)(n+1)^2(n+2)^2t_{n+3} - 3(4n^3+17n^2+14n-21)(n+1)^2t_{n+2}\\ 
+(12n^5 + 87n^4+234n^3+177n^2-84n-126)t_{n+1}\\ -(n^3+5n^2+11n-5)(4n+1)(n+3)t_{n}=0,
\end{split}
\end{equation}
with initial conditions $t_1=1, t_2 = 8, t_3 = 63/2$.  

Applying the methods of \citet{Adams,Birkhoff}, we substitute trial solutions of the form $n^\sigma \delta^n$ and then $n^\rho$ into \eqref{eq:qpadding} and determine the values of $\sigma$, $\delta$ and $\rho$ for which the leading term is zero.  We find that $\delta= 1$ and then $\rho = 1,2,3$ and therefore conclude that $t_n$ grows asymptotically with the cube of $n$.  Solving the recurrence numerically for the given initial conditions we find that $\lim_{n \to \infty} n^{-3}t_n \doteq 5.7946$. This can be compared with the estimate $5.76$ obtained by simulation in \citet{BNK}.  

For the fourth moment $f_n = E(X^4_n) = a_{n}^{[4,0]}$ we have corresponding equations
\begin{eqnarray*}
a_{n+1}^{[0,4]} &=& a_{n}^{[0,4]} + 4a_{n}^{[1,3} + 6a_{n}^{[2,2]} + 4a_{n}^{[3,1]} + a_{n}^{[4,0]}, \quad   
a_{n+1}^{[4,0]} = \frac{2}{n} b_{n}^{[4,0]} + \frac{8}{n^2} b_{n}^{[3,1]} + \frac{6}{n^2} c_n^{[2]},\\
a_{n+1}^{[1,3]} &=& \frac{2}{n} a_{n+1}^{[0,4]}, \quad
a_{n+1}^{[2,2]} = \frac{2}{n} b_{n}^{[2,2]} + \frac{2}{n^2} a_{n+1}^{[0,4]},\quad
a_{n+1}^{[3,1]} = \frac{2}{n} b_{n}^{[3,1]} + \frac{6}{n^2} b_{n}^{[2,2]},\\
b_{n+1}^{[2,2]} &=& b_{n}^{[2,2]}\left(1 + \frac{4}{n} + \frac{2}{n^2}\right) + \frac{2}{n}c_n^{[2]} + a_{n+1}^{[2,2]} + 2a_{n}^{[3,1]} + a_{n}^{[4,0]}, \quad b_{n+1}^{[4,0]} = b_{n}^{[4,0]} + a_{n+1}^{[4,0]},\\
c^{[2]}_{n+1} &=& c^{[2]}_{n}\left(1 + \frac{4}{n}\right) + \frac{4}{n^2}b_{n}^{[2,2]} +  a_{n+1}^{[0,4]}, 
\quad b_{n+1}^{[3,1]} = b_{n}^{[3,1]}\left(1+ \frac{2}{n}\right) + a_{n+1}^{[3,1]} + a_{n+1}^{[4,0]},
\end{eqnarray*}     
where $c_n^{[2]} = E\left\{\left(\sum_{k=1}^n X^2_k\right)^2\right\}$. Again, reducing the system to a single recurrence for $f_n$ and applying the methods of \citet{Adams,Birkhoff}, we find that $f_n$ grows asymptotically with the fourth power of $n$.  Solving the recurrence numerically we have $\lim_{n \to \infty} n^{-4}f_n \doteq 31.585$.

An understanding of further properties of the process $(X_n; n \geqslant 1)$ is greatly aided by the content of the following:
\begin{lemma} \label{thm:martingale} Let $M_n = S_n/(n(n+1))$ where $S_n = \sum_{k=1}^n X_k$, then $(M_n; n \geqslant r)$ is a martingale with respect to the increasing sequence of $\sigma$-fields $(\mathcal{F}_n; n \geqslant r)$ generated by the sequence $(X_n)$, or equivalently $(S_n)$. Furthermore, there exists a non-degenerate random variable $M$, such that $M_n$ converges to $M$ almost surely and in mean-square as $n \to \infty$, where 
$$E M = \frac{1}{r(r+1)} \sum_{k=1}^r x_k  \quad \text{and} \quad E (M^2) = \frac16 K(x_1,\dots,x_r).$$
\end{lemma}
\begin{proof}
By conditional expectation 
$$E \left(M_{n+1} | \mathcal{F}_n\right) = \frac{E(S_n + X_{n+1}| \mathcal{F}_n) }{(n+1)(n+2)} = \frac{S_n+2S_n/n}{(n+1)(n+2)} = M_n.$$ The mean of $M_n$, and hence of $M$, follows from \eqref{eq:martmean}. Dividing \eqref{eq:qn+1} by $n^2$, allowing $n \to \infty$, and noting \eqref{eq:qlimit}, yields 
\begin{equation} \label{eq:alimit}
\lim_{n\to \infty}n^{-4} E S^2_n = \frac16 K(x_1,\dots,x_r).
\end{equation}
The existence of this limit ensures that $E M_n^2$ is uniformly bounded for all $n$. The probabilistic limit results then follow from the martingale convergence theorem \citep{Doob}.  
\end{proof}
As an immediate corollary, we remark that 
\begin{equation} \label{eq:Xmart}
E (X_m M_n | \mathcal{F}_m) = \frac{X_m S_m}{m(m+1)}, \quad \text{for $n \geqslant m$}.
\end{equation}
We now turn to consider the auto-covariance properties of the process $(X_n)$, where we have the following result.
\begin{theorem} \label{thm:2}
Let $m,n \to \infty$, with $m \leqslant n$ then, 
\begin{equation}
n^{-2} E(X_m X_n) \to \begin{cases}\frac{2}{3} \theta K & \text{if $m/n \to \theta \in (0,1)$},\\
                                     K& \text{if $m=n$},
                       \end{cases}
\end{equation}
where $K=K(x_1,\dots,x_r)$ is defined in \eqref{eq:K} above.                        
\end{theorem}
\begin{proof}  
By conditional expectation, we have
$$
E (X_{n+1}S_{n+1}) = \frac{2}{n}p_n + q_{n+1}, 
$$ 
so that \eqref{eq:alimit} and theorem \ref{thm:qlimit} give
\begin{equation} \label{eq;xslimit}
\lim_{n\to \infty}n^{-3} E(X_{n+1}S_{n+1}) = \textstyle{\frac13} K.
\end{equation}
Again by conditional expectation, we have
\begin{equation} \label{eq:autocov}
E(X_m X_{n+1}) = \frac{2}{n} E (X_m S_n), \quad n \geqslant m, 
\end{equation}
so that with $m=n$
\begin{equation}
\lim_{n\to \infty}n^{-2} E(X_{n}X_{n+1}) = \lim_{n\to \infty}2 n^{-3} E(X_{n}S_{n}) = \textstyle{\frac23} K.
\end{equation}
Finally, considering $m,n \to \infty$ with $m/n \to \theta$, for some fixed $\theta \in (0,1)$, by \eqref{eq:autocov} and \eqref{eq:Xmart} we have
$$\frac{E (X_m X_n)}{n(n+1)} = \frac{2 E(X_mS_{n-1})}{(n-1)n(n+1)} = \frac{2 E(X_m S_m)}{(n-1)m(m+1)} \to \textstyle{\frac23} \theta K.$$
\end{proof}
\subsection{Sample paths}\label{sec:samplepaths}
We can now give an informal but quite precise description of a typical trajectory of the process $(X_n, n \geqslant 1)$ for large $n$. From lemma \ref{thm:martingale} we see that each trajectory has its own limiting value of $M_n =S_n/(n(n+1))$. These values vary from trajectory to trajectory with $E M^2_n \to  K/6$ as $n \to \infty$. Computer simulations show that, when scaled by $S_n/(n+1)$, the variables $X_{U(n)} + X_{V(n)}$ have approximate probability density $f_W(w) = w e^{-w}, w >0$,  as illustrated in figure \ref{fig:1}. This is not unexpected since a related energy splitting model (also due to Ulam, see \citet{BM}) has $f_W(w)$ as its fixed point density.  Specifically, the process $(X_n, n \geqslant 1)$ can be reformulated as follows. At stage $n+1$, sample from the collection $ \{Y_k, k=1,\dots,n\}$ where $Y_k=X_k/k$, by selecting an index $k$ uniformly from $\{1,\dots,n\}$. Then multiply $Y_k$ by $k/(n+1)$, repeat independently and add the results to obtain $Y_{n+1}$. The analogous splitting model is defined by the distributional equality $W \overset{d}{=} U W_1  + V W_2 $ where $U$ and $V$ are independent variables uniform on $(0,1)$ and $W_1$ and $W_2$ are independent copies of $W$. It is straightforward to show that $f_W(w)$ is the fixed point density and it also follows that $UW_1$ and $VW_2$ are independently exponentially distributed. 
More formally we have the following 
\begin{theorem}\label{thm:weak}
\begin{equation}\label{eq:weak1}\lim_{n \to \infty} P(X_n/(M n)\leqslant x) = \int_0^x w e^{-w} dw, \quad x \geqslant 0.
\end{equation}
Furthermore 
\begin{equation}\label{eq:weak2}\lim_{n \to \infty} P(X_{U(n)}/(M n) \leqslant x) = \int_0^x e^{-w} dw, \quad x \geqslant 0.
\end{equation}
\end{theorem}
\begin{proof} 
Our approach follows that of \citet{rosler1991limit} and \citet{geiger2000new} in their analysis of the Quicksort algorithm and Yaglom's exponential limit law. We make use of the Mallows distance (also known as Vassershtein distance) between the distribution of random variables $X$ and $Y$, namely
$$ d_2(X,Y) = \inf_{\mathscr{C}} \sqrt{E(X-Y)^2},$$
where the infimum is over all couplings of $X$ and $Y$, i.e.\ over all joint distributions with the specified margins. For background see \citet[\S 8]{bickel1981some}. In particular, note that the infimum is always attained. 

Let  $\tilde{Y}(k) = X_k/(k M_{k- 1}), k = 1,2,\dots$ and let $U',V'$ be independently uniform on $(0,1)$ then \eqref{eq:BSU} can be recast as
\begin{equation}\label{eq:Y} \tilde{Y}(n+1) = \frac{X_{\lceil  U'n \rceil}}{(n+1) M_n} + \frac{X_{\lceil V' n \rceil}}{(n+1) M_n}.
\end{equation}

Now let $\beta_k$ be the squared Mallows distance between the distribution of $\tilde{Y}(k)$ and the distribution of $W$, $k \geqslant 1$.  We will always choose a representation  $\tilde{Y}(k)$ that attains the infimum of $E(\tilde{Y}(k)-W)^2$ over all couplings $\mathscr{C}$, so that 
$$\beta_k = d^{\,2}_2(\tilde{Y}(k),W) = E(\tilde{Y}(k)-W)^2, \quad k = 1, 2, \dots.$$
From lemma 8.3 in \citet{bickel1981some} for \eqref{eq:weak1}, it is sufficient to show $\beta_{n} \to 0$ as $n \to 0$.

From \eqref{eq:Y} and using $W \overset{d}{=} U W_1  + V W_2 $ we have
\begin{align}\label{eq:beta1}\beta_{n+1} &= d^{\,2}_2(\tilde{Y}(n+1),W) = d^{\,2}_2(\tilde{Y}(n+1),W_1 U + W_2 V)\nonumber\\
&= d^{\,2}_2\left(\frac{X_{\lceil U' n \rceil}}{(n+1) M_n} + \frac{X_{\lceil V' n \rceil}}{(n+1) M_n},W_1 U + W_2 V\right)\nonumber\\
&\leqslant d^{\,2}_2\left(\frac{X_{\lceil U' n \rceil}}{(n+1) M_n},W_1 U\right) + d^{\,2}_2\left(\frac{X_{\lceil V' n \rceil}}{(n+1) M_n},W_2 V\right)\nonumber\\
&= 2d^{\,2}_2\left(\frac{X_{\lceil U' n \rceil}}{(n+1) M_n},W_1 U\right),
\end{align}
using lemma 8.7 of \citet{bickel1981some} for the inequality, since $X_{\lceil U' n \rceil}/((n+1)M_n)$ and $X_{\lceil V' n \rceil}/((n+1)M_n)$ are iid conditional on $\mathcal{F}_n$, each with the same mean as  $W_1 U$ and $W_2V$.

With the coupling $U' = U$ we then have
\begin{align}\label{eq:beta2}\beta_{n+1} &\leqslant 2 E\Big(\frac{X_{\lceil U n \rceil}}{(n+1) M_n}-W_1 U\Big)^2\nonumber\\ 
&= 2E\big((Y(\lceil U n \rceil)-W_1)\phi_n(U) + W_1 (\phi_n(U) - U)\big)^2\nonumber\\
&\leqslant 2E\big((Y(\lceil U n \rceil)-W_1)\phi_n(U)\big)^2 + 2E\big(W_1^2(\phi_n(U) - U)^2\big) \nonumber\\
&=2 \sum_{k=1}^n E(Y(k)-W)^2 \int_{(k-1)/n}^{k/n}\phi_n^2(u)du + 2E(W^2) E(\phi_n(U)-U)^2\nonumber\\
&=2 \sum_{k=1}^n \beta_k \frac{k^2 M^2_{k-1}}{n(n+1)^2M^2_n}+ 12E\big(\phi_n(U)-U)^2,
\end{align}
where $\phi_n(u) =  \lceil u n \rceil M_{\lceil u n \rceil - 1}/((n+1) M_n), 0\leqslant u \leqslant 1.$

We can write the first term in \eqref{eq:beta2} as the weighted average of $(\beta_k M^2_{k-1})$ multiplied by a term converging to $2/(3M^2)$ as $n \to \infty$, so that for the $\limsup$ we have
$$ \limsup \sum_{k=1}^n  \frac{6 k^2 \beta_k M^2_{k-1}}{n(n+1)(2 n+1)} \frac{n(n+1)(2 n+1)}{3n(n+1)^2M_n^2} \leqslant  \limsup \{\beta_n M^2_{n-1}\} \frac{2}{3M^2} = \frac23 \beta,$$
where $\beta = \limsup \beta_n$. Since the second term in \eqref{eq:beta2} converges to $0$ as $n \to \infty$, by taking the $\limsup$ of both sides of  \eqref{eq:beta2} we have $\beta \leqslant \frac23 \beta$ and hence $\beta_n \to 0$. The assertion \eqref{eq:weak2} follows in the same way via \eqref{eq:beta2} since $W_1 U$ has the standard exponential distribution.
\end{proof} 
Informally it can then be argued that $Y$, the limiting value of $Y_n$, is of the form $WM$ where $M$ is the limiting distribution of $M_n$. As a consequence, the moments of $Y$ should have a simple relation to those of $M$, for example $E(Y^2) = 6 E(M^2)$. 

\begin{figure}[h]
\centering
\centering
\includegraphics[width=0.45\linewidth,angle=90,trim=5cm 1cm 7cm 1cm]{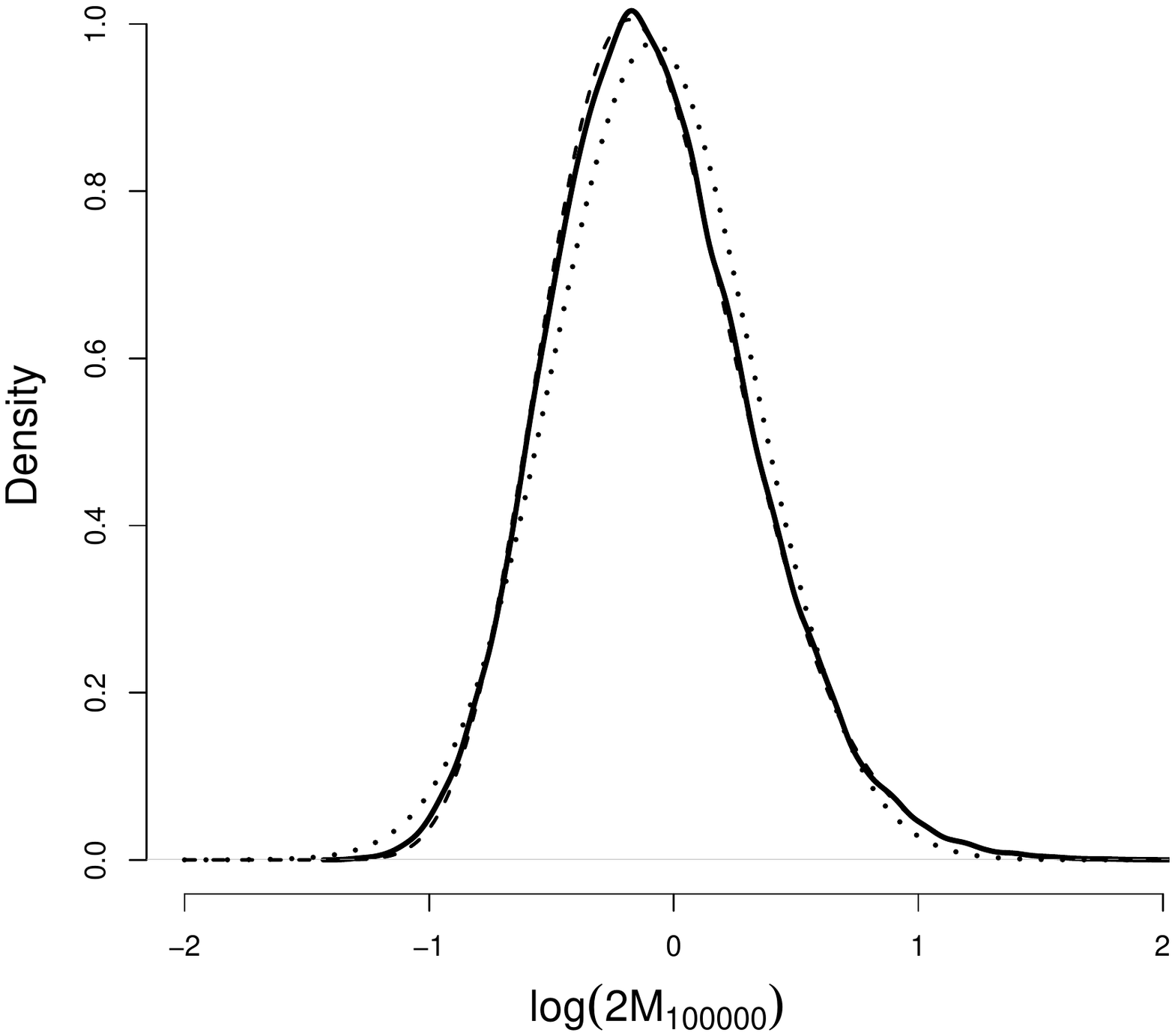}
\includegraphics[width=0.45\linewidth,trim=1cm 4.7cm 1cm 7cm]{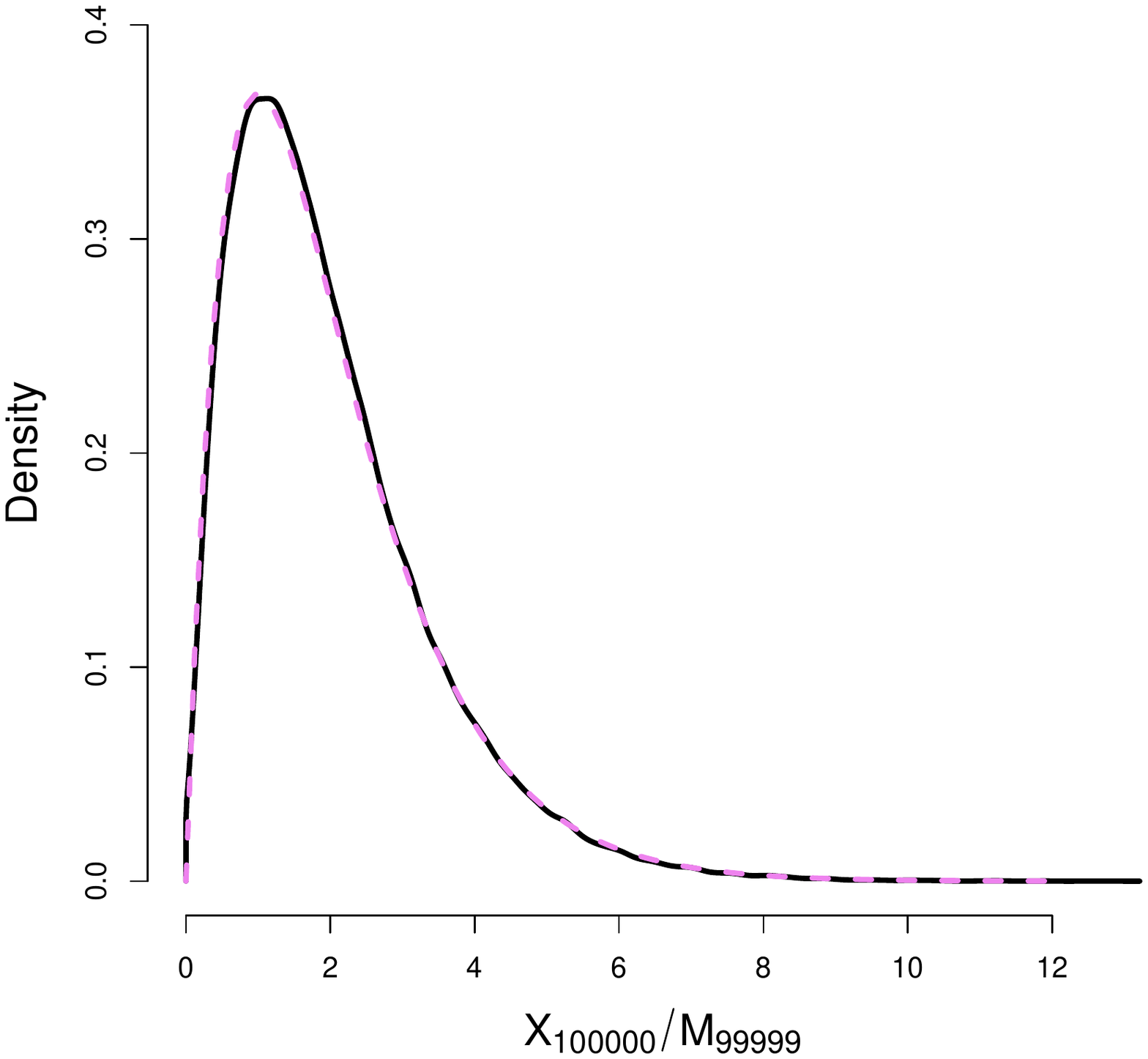}
\caption{Left: simulated density of $\log(2 M_{100000})$ (---) with fitted normal density (...) and gamma density (- - -). Right: simulated density of $X_{100000}/M_{99999}$ (---) with fitted density $f_W(w)$ (- - -); $10^4$ independent realisations.} \label{fig:1}
\end{figure}

Furthermore, we can write 
$$M_{n+1}=\frac{S_{n+1}}{(n+1)(n+2)} = \frac{S_n + X_{U(n)} +X_{V(n)}}{(n+1)(n+2)} 
= M_n\left\{1+ \frac{W_n-2}{n+2}\right\},$$   
where $W_n = \left(X_{U(n)}+X_{V(n)}\right)/M_n,$ and since we have empirical evidence that $(W_n)$ are independently distributed from the same distribution, we can anticipate that the limiting distribution of $M_n$ will be approximately log-normal or, more generally, in the log-gamma family. Figure \ref{fig:1} shows the estimated density of $\log (2 M_n)$, simulated from the initial condition $x_1=1$, compared with fitted gamma and normal densities. (The factor of $2$ is introduced for convenience, so that the mean of $2M_n$ is 1 with this initial condition.) The log-gamma density is seen to provide an excellent fit. 
As a more rigorous test, we can compare the numerically determined moments of $2M$ with those of the candidate distributions.  Using $\mu_k = E\{(2M)^k\} = E\{(2Y)^k\}/ E(W^k)$ for $k=1,2,3,4$ and the moments of $Y$ obtained in the previous section we find $\mu_1=1, \mu_2 \doteq 1.225,\mu_3 \doteq 1.932,\mu_4 \doteq 4.211$. The fourth moment of a log-gamma distribution fitted by the first three of these moments is $4.194$ which is within half a percent of the value $\mu_4$.

\section{The \texorpdfstring{$p$}{p}-adding process in discrete time}\label{sec:3}
We now consider a simple modification of the process defined in \eqref{eq:BSU}, where history-dependent updates occur randomly and independently with probability $p$, where $p<1$.  The new process is as follows.

\medskip
{\sc Definition 3.1}. Let $(J_n,n\geqslant 1)$ be a sequence of independent Bernoulli variables each with success probability $p$ and let $(U(n))$ and $(V(n))$ be sequences of independent variables (also independent of $(J_n)$) such that for any given $n$, $U(n)$ and $V(n)$ are each uniformly distributed on $\{1,\dots,n\}$. The $p$-adding process with fixed initial condition $(X_1=x_1,\dots,X_r=x_r)$ is defined by
\begin{equation}\label{eq:p_add}  
X_{n+1} = J_n[X_{U(n)} +X_{V(n)}] +(1-J_n)X_n, \quad n\geqslant r.
\end{equation}
\begin{theorem}\label{thm:paddingmean}
The $p$-adding process has mean 
\begin{equation} \label{eq:paddingmean}
E X_n = (\nu + p n)\left\{\frac{x_r}{\nu +p r} + \frac{C p r}{\nu^{r-1}}\sum_{k=r}^{n-1} \frac{\nu^{k-1}}{k(\nu+p k)(1 + pk)}\right\},\quad n \geqslant r,
\end{equation} 
where $\nu=1-p$, $C= 2 s_r (\nu+pr)/r - 
(1+\nu+pr)x_r$ and $s_r= \sum_{k=1}^r x_k$, with the convention that the summation in \eqref{eq:paddingmean} is zero when the upper limit is less than the lower.
\end{theorem}
\begin{proof}
Let $m_n= E X_n$, then by the usual conditioning arguments
\begin{equation} \label{eq:paddingint}
m_{n+1} = \nu m_n + \frac{2 p}{n} \sum^n_{k=1} m_k,\quad n\geqslant r,
\end{equation}
which can be recast as the difference equation
\begin{equation} \label{eq:paddingDE}
(n+1)m_{n+2} -[n(1+\nu)+ p+1]m_{n+1} + n\nu m_n = 0, \quad n\geqslant r.
\end{equation}
By inspection, $\nu+pn$ is seen to be a solution and the Casoratian can be shown to be $\nu^{k-1}/k$.  The general solution is then 
$$
(\nu+pn)\left\{A + B \sum_{k=1}^{n-1} \frac{\nu^{k-1}}{k(\nu+p k)(1 + pk)}\right\},n\geqslant r,
$$
where $A$ and $B$ are arbitrary constants.  The solution \eqref{eq:paddingmean} then follows from the initial conditions $m_r = x_r$ and $m_{r+1} = \nu x_r + 2p s_r/r$. \end{proof}
From \eqref{eq:paddingmean}, since the partial sum has a finite limit, we see that $m_n$ grows linearly with $n$ as $n \to \infty$.  For the second moment we have
\begin{theorem} \label{thm:qpaddinglimit}
\begin{equation}\label{eq:qpaddinglimit}
\frac{E X_n^2}{n^2} \to K(p,x_1,\dots,x_r), \quad  \text{as $n\to \infty,$}
\end{equation}
where $K$ is a function of $p$ and $x_1,\dots,x_r$;  equal to $K$ in \eqref{eq:K} when $p=1$.
\end{theorem}
\begin{proof}
Define $S_n = \sum_{k=1}^n X_k$ and let $p_n = E S_n^2$, $q_n= E X^2_n$, $w_n = E X_n S_{n-1}$ and $t_n=\sum_{k=1}^n q_k$. Using the usual conditional expectation arguments, we have
\begin{equation} \label{eq:qw}
 q_{n+1} = (1-p) q_n +p\left[\frac2n t_n +\frac2{n^2}p_n\right], \quad w_{n+1} = (1-p)(w_n +q_n) +\frac{2p}n p_n,\quad n \geqslant r,
\end{equation}
with the additional identities $t_{n+1} = t_{n} + q_n$ and $p_n = p_{n-1} + 2 w_n + q_n$.  Omitting the details for the sake of brevity, this system of recurrences can be reduced to the single fourth order linear difference equation for $q_n$. 
\begin{equation}\label{eq:qpadding}
\begin{split}
(n+3)^2 q_{n+4} + [(2p-4)n^2+(4p-18)n-3p-21]q_{n+3} \\
+[(6-6p+p^2)n^2+(18-8p-2p^2)n+15+2p^2]q_{n+2}\\
- (1-p)[(4-2p)n^2+(2p+6)n+3]q_{n+1} +(1-p)^2 n^2q_n =0. 
\end{split}
\end{equation}
As before, we refer to \citet{Adams,Birkhoff} and substitute trial solutions of the form $n^\sigma \delta^n$ and then $n^\rho$ into \eqref{eq:qpadding}.  By considering the leading terms in the resulting expressions, we find that $\delta= 1, 1-p$ and then $\rho = 1,2$ and therefore conclude that $q_n$ grows quadratically as $n \to \infty$.
\end{proof}
\subsection{Numerical results}
We have investigated the behaviour of $K(p,x_1,\dots,x_r)$ numerically for various values of $p$ in the case $x_r=r=1$.  For comparison purposes, we rescale time so that for each $p$ jumps occur at mean rate 1. On this time scale the limiting constant is  $K(p,1)/p^2$. The results are illustrated in figure \ref{fig:2}.  The exact value at $p=1$ is given in theorem \ref{thm:qlimit}.   Theorem \ref{thm:8} for the continuized model provides the limiting value as $p\to0$, namely $\cosh(\pi \sqrt{7}/2)/(4 \pi) \doteq 2.53961.$

Note that a simple lower bound for $K(p,1)/p^2$ in all cases can be obtained from the observation that $E S^2_n \geqslant (E S_n )^2$. Then, since $n^{-2}S_n \to \frac12 p$ from \eqref{eq:paddingmean} and $n^{-4}E S^2_n \to \frac16 K(p,1)$, as will be shown in theorem \ref{thm:5} below, it follows that $K(p,1)/p^2 \geqslant 1.5$.  A similar calculation in terms of $E X^2_n$ yields the uniformly worse lower bound of $1$.   

Numerical values of the product moment $n^{-2} E(X_m X_n)$ for the $p$-adding process are shown in figure \ref{fig:3}. A simple limiting pattern emerges with a discontinuity at $m=n$,  at which the value drops by one third.  This phenomenon is explained in theorem \ref{thm:5} below.
\begin{figure}[h]
\centering
\begin{minipage}[t]{0.47\linewidth}
\centering
\includegraphics[width=0.92\linewidth,angle=90,trim=5cm 1cm 7cm 1cm]{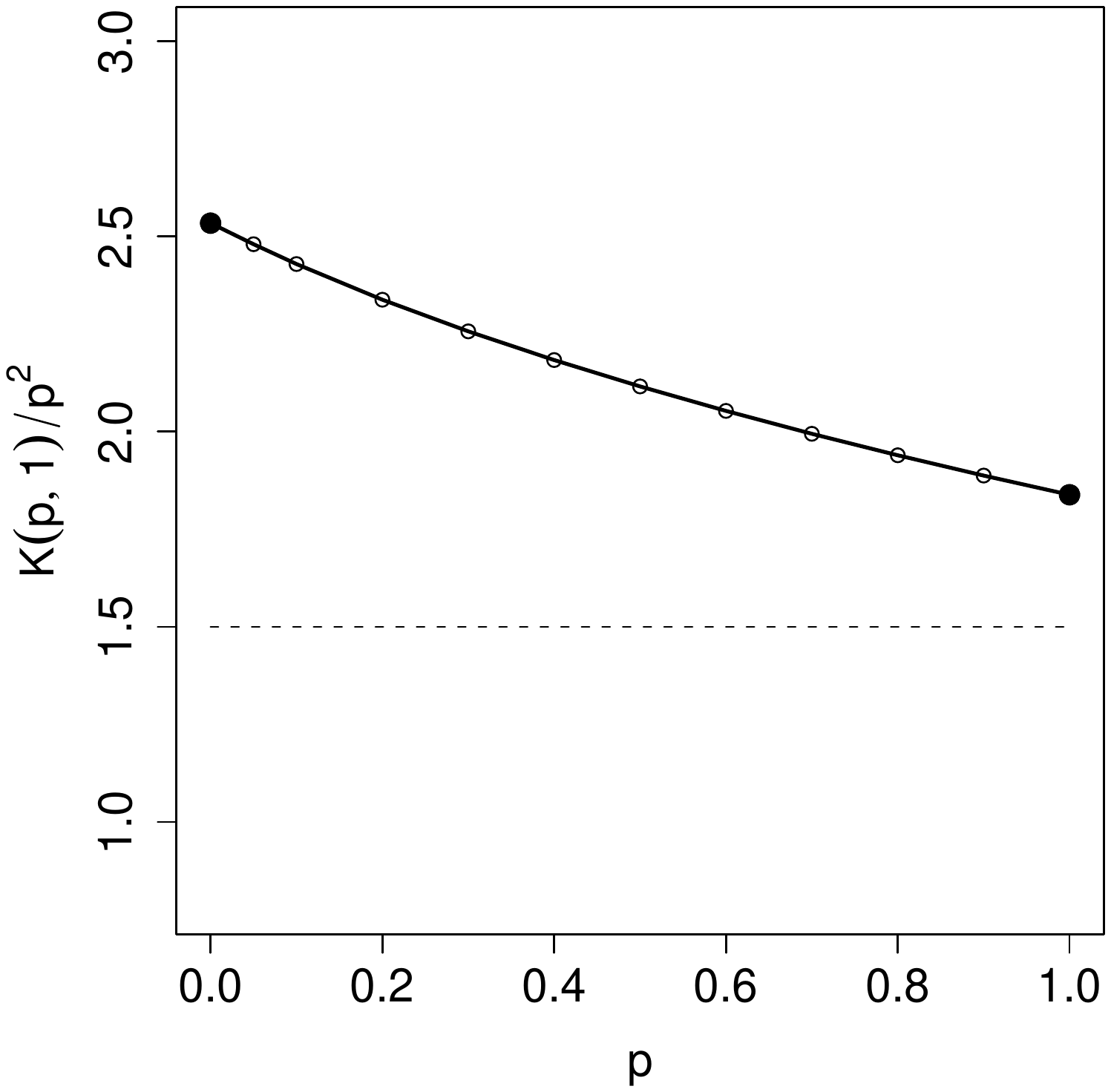}
\caption{Asymptotic growth rate of the second moment, $q_n$.} \label{fig:2}
\end{minipage}%
\hspace{0.5cm}%
\begin{minipage}[t]{0.47\linewidth}
\centering
\includegraphics[width=0.92\linewidth,angle=90,trim=5cm 1cm 7cm 1cm]{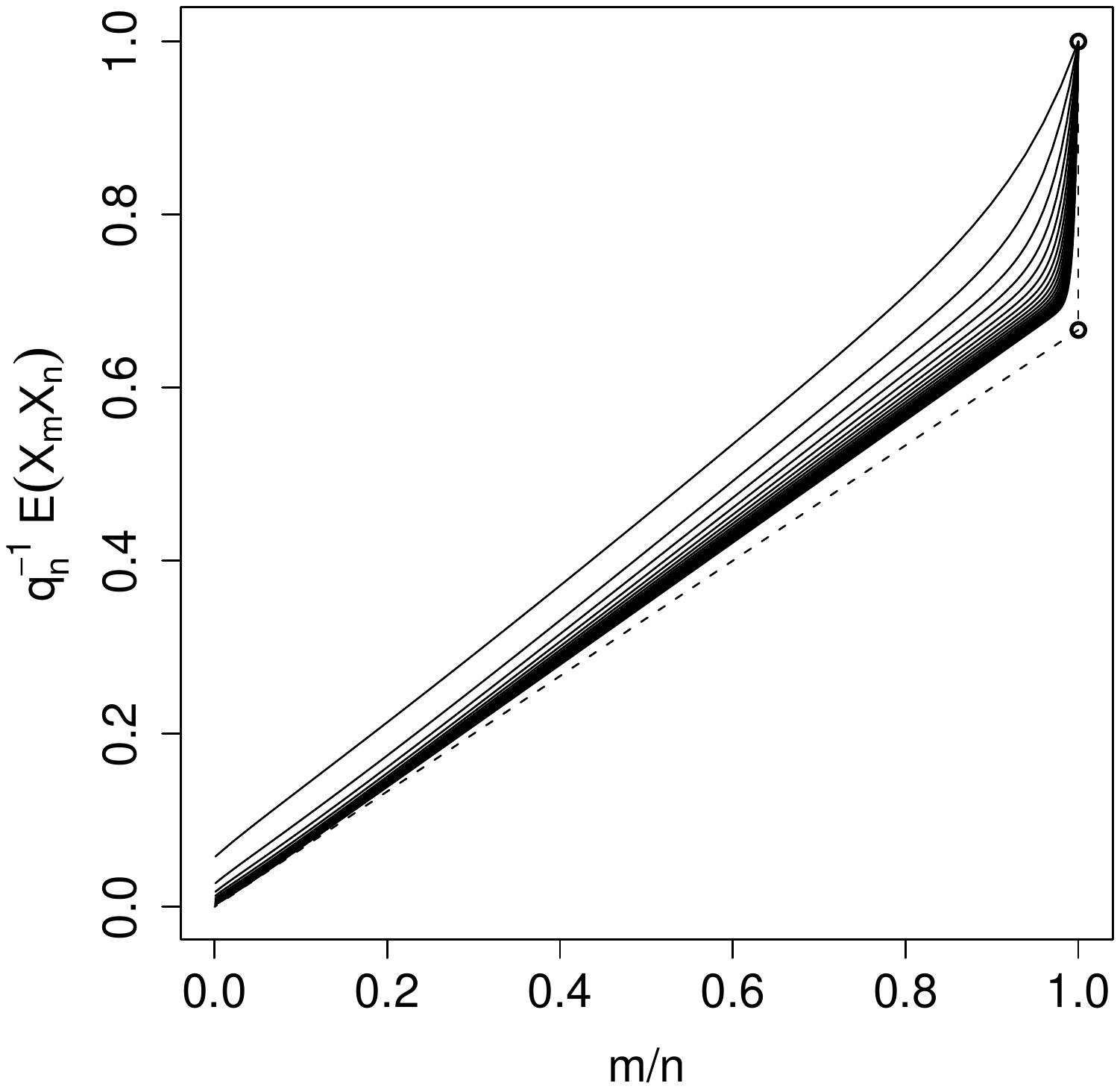}
\caption{Convergence of the product moment, $n=50,100,\dots,1000$; $p=0.2$.} \label{fig:3}
\end{minipage}
\end{figure}

In considering the product moment of the basic adding process we are aided by the existence the martingale that yields \eqref{eq:Xmart}.  Similar conclusions can be drawn for the $p$-adding process, as follows.
\begin{theorem}\label{thm:5}
 Let $m,n \to \infty$, with $m \leqslant n$ then the limiting product moment of the $p$-adding process $(X_n)$ with $p<1$ is given by
\begin{equation}
n^{-2} E(X_m X_n) \to \begin{cases}\frac{2}{3} \theta K & \text{if $m/n \to \theta \in (0,1)$},\\
                                     K& \text{if $m=n$},
                       \end{cases}
\end{equation}
where $K=K(p,x_1,\dots,x_r)$ is defined in \eqref{eq:qpaddinglimit} above.                         
\end{theorem}
\begin{proof}  
First note that as a consequence of theorem \ref{thm:qpaddinglimit}, we have $t_n/n^3 \to \frac13 K$. Dividing the first equation in \eqref{eq:qw} by $n^2$, we then have $p_n/n^4 \to \frac16 K$ and dividing the second equation by $n^3$ we have $w_n/n^3 \to \frac13 K$.  By the usual conditioning arguments we also have
$E (X_{n+1}S_{n+1}) = \nu [E(X_n S_n) + q_n] +2p n^{-1} p_n$ and, dividing both sides by $n^3$, we have  $n^{-3} E(X_n S_n) \to \frac13 K$ as $n \to \infty$
 
Now let $c_{m,n} = E (X_m X_n)$ then, as in \eqref{eq:paddingint}, we have 
\begin{equation}\label{eq:paddingcov}  
c_{m,n+1} = \nu c_{m,n} + \frac{2 p}{n} \sum^n_{k=1} c_{m,k},\quad n \geqslant m \geqslant r,
\end{equation}
with solution
$$
c_{m,n}= (\nu+pn)\left\{A + B \sum_{k=r}^{n-1} \frac{\nu^{k-1}}{k(\nu+p k)(1 + pk)}\right\},\quad n\geqslant m \geqslant r,
$$
where $A$ and $B$ can be determined from the initial conditions $c_{m,m} = q_m$ and $c_{m,m+1} = \nu q_m +2 p m^{-1}\sum_{k=1}^{m} c_{m,k}$. Thus
\begin{equation} \label{eq:paddingcovsol}
c_{m,n}= (\nu+pn)\left\{\frac{c_{m,m+1}}{\nu+p(m+1)} \frac{H_{n-1}}{H_{m}} - \frac{q_m}{\nu+p m}\frac{H_{n-1}-H_{m}}{H_m}\right\}, \quad n\geqslant m \geqslant r,
\end{equation}
where $H_n = \sum_{k=r}^{n} \nu^{k-1}/[k(\nu+p k)(1 + pk)].$

Returning to \eqref{eq:paddingcov} with $n=m$ we have $c_{m,m+1} = \nu q_m + 2 p n^{-1} E(X_m S_m)$. Dividing by $m^2$ and using the earlier limit results we thus have $c_{m,m+1}/m^2 \to \frac23 K$.  Finally dividing \eqref{eq:paddingcovsol} by $n^2$, letting both $m \to \infty$ and $n \to \infty$ so that $m/n \to \theta < 1$ and noting that $\lim_{n \to \infty} H_n < \infty$,  we have $c_{m,n}/n^2 \to \frac23 \theta K$, as claimed.
\end{proof}
For the discrete time $p$-adding process we have the following:
\begin{lemma} \label{thm:p-martingale} Let 
$$M_n = \frac{p A_n}{n(1+np)}S_n + \frac{\nu  B_n}{(n+1)(2+np)}X_n, \quad \nu = 1-p,$$ 

where $A_n = \sum_{k=n}^\infty a_n/a_k$ with $a_k = \nu^{-k} k(\nu+kp)(1+kp)$ and $B_n = \sum_{k=n}^\infty b_n/b_k$ with 
$b_k = \nu^{-k} k(k+1)(\nu+kp+1)(2+kp)$
then $(M_n; n\geqslant r)$ is a martingale with respect to the increasing $\sigma$-fields $(\mathcal{F}_n;n\geqslant r)$ generated by $(X_n)$ or equivalently $(S_n)$. Furthermore, there exists a non-degenerate random variable $M$ with finite variance, such that $M_n$ converges to $M$ almost surely and in mean-square as $n \to \infty$, where 
$$E M = \frac{pA_r}{r(1+rp)}\sum_{k=1}^r x_k + \frac{\nu B_r x_r}{(r+1)(2+rp)}.$$  
Note that both $A_n$ and $B_n$ converge to $1/p$ as $n \to \infty$.
\end{lemma}
\begin{proof} Consider the sequence $(\alpha_n S_n + \beta_n X_n, n\geqslant r)$. In order for this to be a martingale we require that
$E (\alpha_{n+1} S_{n+1} + \beta_{n+1} X_{n+1} | \mathcal{F}_n) =  \alpha_n S_n + \beta_n X_n, n\geqslant r$. Referring to \eqref{eq:p_add} we have
$$\alpha_{n+1}\left(S_n + \frac{2p}{n}S_n + (1-p)X_n \right) + \beta_{n+1}\left(\frac{2p}{n}S_n + (1-p)X_n\right) = \alpha_n S_n + \beta_n X_n.$$ Equating the coefficients of $X_n$ and $S_n$ gives the pair of difference equations:
$$\beta_n= (1-p)(\alpha_{n+1} + \beta_{n+1}), \quad \alpha_n = \alpha_{n+1}(n+2p)/n + \beta_{n+1} 2p/n.$$ 
Eliminating $\beta_n$ to produce a second order difference equation for $\alpha_n$ and proceeding as in the solution of \eqref{eq:paddingDE} we have
\begin{align*}
\alpha_n &= \nu^{-n} (\nu+np)\left\{C_1+C_2 \sum_{k=1}^{n-1} \frac{\nu^k}{k(\nu+kp)(1+kp)}\right\}\\
&=\nu^{-n} (\nu+np)\left\{C_1+C_2\Phi_p - C_2 \sum_{k=n}^\infty \frac{\nu^k}{k(\nu+kp)(1+kp)}\right\},
\end{align*}
where $\Phi_p = \sum_{k=1}^\infty \nu^k/(k(\nu+kp)(1+kp))$. To find a positive solution that decreases with $n$ we start by setting $C_1 = -C_2 \Phi_p$. Rearranging terms then gives
$$\alpha_n = \frac{-C_2}{n(1+np)}\sum_{k=n}^\infty \frac{\nu^{k-n}n(\nu+np)(1+np)}{ k(\nu+kp)(1+kp)} = \frac{-C_2}{n(1+np)} A_n,$$
and we can now see that $A_n \to 1/p$ as $n \to \infty$.  The solution for $\beta_n$ follows similarly. Taken together the constants in the solutions can be determined to solve the original pair of difference equations.  

Finally from the calculations in the proof of theorem \ref{thm:qpaddinglimit}, we see that $E M_n^2$ is uniformly bounded and so the probabilistic limit follows from the martingale limit theorem  \citep{Doob}.\end{proof}

It follows easily that we have this result, paralleling that of lemma \ref{thm:martingale} :
\begin{corollary}  $ S_n/(n(n + 1))$  converges to $pM$ almost surely and in mean-square as $n \to \infty$, where $M$ is as defined in lemma \ref{thm:p-martingale}.
\end{corollary}
\begin{proof}  From theorem \ref{thm:qpaddinglimit}, we have that $E X_n^2/n^2 \to K(p, x_1, . . . , x_r)$, as $n \to \infty$, and hence $E X_n^2/n^4 \to 0$  as $n \to \infty$, so that $X_n/n^2$ converges in m.s.\ to $0$. Also, by Chebyshov's inequality, $P(X_n/n^2 > a) < EX_n^2/(a^2 n^4)  \sim K/(a^2 n^2)$ for $a>0$ as $n \to \infty$, and the convergence of $\sum (1/n)^2$ implies the a.s. convergence of $X_n/n^2$ to $0$, by the first Borel-Cantelli lemma, as $a$ is arbitrarily small.

Since $B_n$ converges to $1/p$ as  $n \to \infty$, it follows that the second term $\nu B_n X_n/{(n + 1)(2 + n p)}$ in the definition of $M_n$ converges a.s. and in m.s.\ to $0$ and since the sum of two convergent sequences of random variables converges to the sum of the limiting variables both in m.s.\ and almost surely,  the assertion of the corollary follows immediately, when we note that $A_n$ converges to $1/p$ as $n \to \infty$.\end{proof}
The conclusions of section \ref{sec:samplepaths}, about the sample paths of Ulam's base process, are now seen to transfer in just the same way to the $p$-adding process, {\it mutatis mutandis}.  The existence of the convergent martingale was crucial in this.

\section{The continuized adding process}\label{sec:simplecont}

A familiar method for gaining insight into many discrete-time processes is to consider analogous problems in continuous time.  And of course, such processes are of natural interest in their own right. In this case the underlying idea is that the jumps of the discrete process $(X_n)$ should take place at the jump instants of a Poisson process $(N(t))$;  the process $(X_n)$ is then said to be subordinate to $(N(t))$.  Such continuized (or Poisson-regulated) processes have been used previously in analysing other history dependent random sequences \citep{CS} and are also discussed by \citet{Feller}.  We define the continuized random adding process thus:

\medskip
{\sc Definition 4.1}. Let $(T_r; r \geqslant 1)$ be the successive jump times of a Poisson process $(N(t), t > \tau
)$ where $\tau \geqslant 0 $ and $N(\tau) =0$. For notational convenience let $T_0= \tau$.  Without essential loss of generality, we will take the Poisson intensity $\lambda$ to be $1$.  Let $(U_r; r \geqslant 1)$ and $(V_r; r \geqslant 1)$ be independent sequences of independent random variables, such that $U_r$ and $V_r$ are uniformly distributed on $[0,T_r]$. With initial conditions $X(t) = x(t), 0 \leqslant t \leqslant \tau$, the process is defined by
\begin{align} 
X(t) &= X(T_{r-1}) \quad \text{for} \quad T_{r-1} \leqslant t < T_r, \nonumber\\
X(T_r) &= X(U_r) + X(V_r).
\end{align}

Note that many, more general, constructions are possible, in that 
\begin{itemize}
\item[(a)] we could permit $U_r$ and $V_r$ to have a distribution other than uniform,
\item[(b)] we could consider weighting factors so that
$$X(T_r) = A X(U_r) + B X(V_r),$$
where $A$ and $B$ are constants, or even random variables,
\item[(c)] the regulating Poisson process could be non-homogeneous, of rate $\lambda(t)$.
\end{itemize}
We return later to consider some of these more general problems.  

For the process $(X(t))$ of definition 4.1, we have this
\begin{theorem} \label{thm:7} 
Let $m(t) = E X(t)$ be the mean of $X(t)$, then for $t \geqslant \tau > 0$
\begin{equation}\label{meanDEsoln} m(t) = (1+t) \left(\frac{x(\tau)}{1+\tau} + C \int_\tau^t \frac{e^{-y}}{y(1+y)^2} dy\right),\end{equation}
where 
$$\frac{C e^{-\tau}}{\tau(1+\tau)} + \frac{2+\tau}{1+\tau} x(\tau) = \frac{2}{\tau} \int_0^\tau x(u) du.$$
When $\tau=0$, and $x(0)=1$, this yields 
\begin{equation}\label{meanDEspecial} m(t)=1+t. \end{equation}
\end{theorem}
\begin{proof}
for small $h>0$, let $\mathcal{I}_{h,t}$ be the indicator of the event that $N(t+h) = N(t).$ Then by conditional expectation, for $t \geqslant \tau$, 
$$m(t+h) = E\{E[X(t+h)|\mathcal{I}_{h,t}]\} = (1-h)m(t) + h E\{X(U)+X(V)\} + o(h),$$
where $U$  and $V$ are uniformly and independently distributed over $[0,t]$.  Hence
$$m' + m = \frac{2}{t} \int_0^t m(u) du.$$
It follows that 
\begin{equation}\label{eq:mddot}
 t m'' + (1+t) m' - m =0,
\end{equation}
where $m'$ and $m''$ are the first and second derivatives of $m(t)$.

By inspection, $m(t) = 1+t$ is a particular solution of \eqref{eq:mddot}.  The complete solution  \eqref{meanDEsoln} follows routinely, on applying the initial conditions 
$$m(\tau) = x(\tau) \quad \text{and} \quad m'(\tau) = - x(\tau) + \frac{2}{\tau} \int_0^{\tau} x(u) du.$$
\end{proof}
We observe that the special case \eqref{meanDEspecial} essentially reproduces the behaviour of the discrete adding process started at $X_1=1$. 

For the second moment $q(t) = E(X^2(t))$, we have 
\begin{theorem} \label{thm:8} As $t \to \infty$, $q(t)$ grows quadratically with $t$.  In particular, when $\tau=0$ and $x(0)=1$, we have $q(t)/t^2 \to \cosh(\pi \sqrt{7}/2)/(4 \pi) \doteq 2.53961,$ as $t \to \infty$. The second moment is seen to have the same quadratic asymptotic growth behaviour as that of the discrete time processes, but with a larger constant; as perhaps is to be expected intuitively.  
\end{theorem}
\begin{proof}
Conditioning on events of the Poisson process $(N(t))$ during the interval $(t, t+h)$, as above, gives
\begin{equation} 
q' + q = E\{X^2(U)\} + E\{X^2(V)\} + 2 E\{X(U) X(V)\},
\end{equation}
where $U$ and $V$ are independently uniform, so that
\begin{equation} \label{eq:qdot}
q' + q = \frac{2}{t} \int^t_0 q(u) du +\frac{2}{t^2} \int_0^t \int_0^t c(u,v) du dv,
\end{equation}
where $c(u,v) = E\{X(u) X(v)\}.$

In addition, for $u<t$, by similar conditioning, we have 
\begin{equation} \label{eq:cdot1}
\frac{\partial c(u,t)}{\partial t} + c(u,t) = \frac{2}{t} \int_0^t c(u,y) dy,
\end{equation}
and, for $v < t$,
\begin{equation} \label{eq:cdot2}
\frac{\partial c(t,v)}{\partial t} + c(t,v) = \frac{2}{t} \int_0^t c(x,v) dx.
\end{equation}
Now define $Q(t) = \int_0^t \int_0^t c(u,v) du dv$, with first derivative
\begin{equation} \label{eq:Qdot}
Q' = \int_0^t c(u,t) du + \int_0^t c(t,v) dv.
\end{equation}
Differentiating again and substituting from \eqref{eq:cdot1},\eqref{eq:cdot2} and \eqref{eq:Qdot}, we obtain
\begin{equation} \label{eq:Qddot}
Q'' + Q' = 2 q + \frac{4}{t} Q.
\end{equation}
Eliminating $Q$ from \eqref{eq:Qddot} and \eqref{eq:qdot} gives 
\begin{equation}\label{eq:contqDE}
t^2 q^{(\text{iv})} + (6 t + t^2) q''' +(6+4t +t^2) q'' - (6+2t) q' + 2 q =0.
\end{equation}
Following \citet{Erdelyi}, we determine the asymptotic growth rate of $q(t)$, as $t \to \infty$, by substituting trial solutions of the form $q_1 = t^\sigma e^{\delta t}$ and $q_2  = t^\rho$; this procedure yields the leading term in an asymptotic expansion developed in inverse powers of $t$.  For $q_1$, we find that the leading term is zero when $\delta^4 + 2 \delta^3 + \delta^2 = 0$ whence $\delta = 0$ or $\delta = -1$.  We therefore consider substitutions of the form $q_2$, which then gives $(\rho-2)(\rho -1) = 0$. Thus $q(t) \sim K t^2$, as asserted, where $K$ is a constant depending on the initial conditions $\{x(t), 0 \leqslant t \leqslant \tau\}$ .

For the base case where $\tau=0$ and $x(0)=1$, we can determine the coefficient $K$ explicitly.  We start by defining the transform $g(s) = s^{-1} \int_0^\infty e^{-t/s} q(t) dt$ for $s>0$.  The function $g$ is well defined since we have established that $q(t) \sim K t^2$. The asymptotic behaviours of $g$ and $q$ are related by a Tauberian theorem \citep[page 220]{Feller}, namely 
\begin{equation}\label{eq:Tauber}
q(t) \sim K t^\alpha,\; \text{as $t \to \infty$} \quad \text{if and only if} \quad g(s) \sim K s^\alpha \Gamma(\alpha+1),\; \text{as $s \to \infty$}.
\end{equation}
For the base case, using \eqref{eq:qdot} and \eqref{eq:Qdot}, the initial conditions for $q$ are found to be $q(0) = 1$, $q'(0) =3$, $q''(0) = 8/3$ and $q'''(0) = 4/9$.  Applying the transform to \eqref{eq:contqDE}, after some reduction, we have 
\begin{equation}\label{eq:contgDE}
s(s+1)^2 g''(s) + 2(1-s^2) g'(s) + 2(s-3)g(s) = 0,      
\end{equation}
with $g(0)=1$ and $g'(0) = 3$.  The method of Frobenius provides solutions for $g(s)$ of the form $C_1 P(s) + C_2 R(s)$ where 
\begin{eqnarray*}
P(s) &=& 1 + 3 s + 8s^2/3  + \dots,\\
Q(s) &=& \log(s)[2 + 6 s + 16 s^2/ 3 + \dots] + s^{-1} [ 1+ 4 s + 2 s^2 + \dots].
\end{eqnarray*}
Clearly $P(s)$ is the required solution of \eqref{eq:contgDE} but, expressed as a power series, it provides no immediate access to the asymptotic growth of $g(s)$. An alternative pair of solutions can be found by shifting to the singular point $s=-1$, i.e.\ by defining $g(s) = u(1+s)(1+s)^{2+\beta}$ and considering the differential equation satisfied by $u$.  Taking $\beta$ to be $\frac12(1-i \surd 7)$ or its complex conjugate, we find
\begin{equation}
w(w-1) u''(w) + 2[(\beta+1)w - \beta]u'(w) + 2\beta^2 u(w) = 0,
\end{equation}
a hypergeometric differential equation \citep[\S 15.5.1]{AS} with a solution $G(\beta,w) = F(\beta,\beta+1, 2\beta, w)$ where 
$F$ is the hypergeometric function defined in \citep[\S 15.1.1]{AS}.  It follows that \eqref{eq:contgDE} has solution
\begin{equation}\label{eq:gsoln}
g(s) = B_1(1+s)^{2+\beta} G(\beta,1+s) +B_2(1+s)^{2+\bar{\beta}} G(\bar{\beta},1+s),
\end{equation}
where $\bar{\beta}$ is the complex conjugate of $\beta$ and $(B_1,B_2)$ are complex constants chosen so that $g(s) = P(s)$.  

First note that the general Frobenius solution has the property that $s[C_1P(s)+C_2Q(s)]$ converges to $C_2$ as $s \to 0$. So, in order that $C_2 = 0$ we must have $s g(s) \to 0$ as $s\to 0$ in \eqref{eq:gsoln}.  Furthermore using \citet[\S 15.3.3]{AS} we have 
$- s G(\beta ,1+s)= F(\beta,\beta-1, 2\beta,1+s),$
and using \citet[\S 15.1.20]{AS}, the limiting value of the right-hand side of this equation, as $s \to 0$,  is given by 
$A(\beta) = \Gamma(2\beta)/[\Gamma(\beta)\Gamma(\beta+1)].$  
It follows that as $s \to 0$, $- s g(s) \to B_1 A(\beta) +B_2 A(\bar{\beta}) = 0$ and hence $B_2/B_1 = -A(\beta)/A(\bar{\beta})$. 

The solution we require is then 
\begin{equation}\label{eq:B0}
B_0[A(\bar{\beta})(1+s)^{2+\beta} G(\beta,1+s) - A(\beta)(1+s)^{2+\bar{\beta}} G(\bar{\beta},1+s)],
\end{equation}
where $B_0$ has to be found so that $g(s)$ satisfies the initial condition  $g(0) = 1$. From \citet[\S 15.3.12]{AS} the constant term in the expansion  of  $G(\beta, 1+s)$ about $s=0$, i.e.\ the term with $n=0$, is given by 
$$\frac{\Gamma(2 \beta)}{\Gamma(\beta-1)\Gamma(\beta)} [\psi(\beta) + \psi(\beta+1) - \psi(1) - \psi(2)] = H(\beta) \quad \text{(say)},$$
where $\psi(z) = d/{dz}\log \Gamma(z).$  Thus the constant term on expanding  $(1+s)^{2+\beta}  G(\beta, 1+s)$ is $H(\beta) - (2+\beta) A(\beta)$, and after some simplification, the constant term in \eqref{eq:B0} is 
$$B_0\{A(\bar{\beta})[H(\beta)- \beta A(\beta)] - A(\beta) [ H(\bar{\beta}) - \bar{\beta} A(\bar{\beta})]\} = B_0 i \surd 7,$$
from which it follows that $B_0 = (i \surd 7)^{-1}.$  

We now have the required solution explicitly in the form \eqref{eq:B0}.  It remains to determine the asymptotic behaviour as $s \to \infty$.  From \citet[\S 15.3.4]{AS} we have $G(\beta, 1+s) = (-s)^{-\beta} F(\beta,\beta-1,2 \beta, s^{-1}(1+s))$, so that,  from the definition of $A(\beta)$,
$$ \lim_{s \to \infty}(1+s)^\beta G(\beta, 1+s) = \lim_{s \to \infty} \left(\frac{1+s}{-s}\right)^\beta A(\beta) =  i e^{\frac12 \pi \surd 7} A(\beta).$$
Consequently using the form \eqref{eq:B0} 
\begin{equation}
\lim_{s \to \infty} \frac{g(s)}{(1+s)^2} =  \frac{A(\bar{\beta}) A(\beta)}{ \surd 7}\left[ e^{\frac12 \pi \surd 7} - e^{-\frac12 \pi \surd 7} \right]
= \frac{\cosh^2(\frac12 \pi \surd 7)}{2 \pi \sinh(\pi \surd 7 )} \times \textstyle{2 \sinh(\frac12 \pi \surd 7)}.
\end{equation}
Therefore $g(s)/s^2 \to \cosh( \frac12 \pi \surd 7) /(2 \pi)$ and from the Tauberian relation \eqref{eq:Tauber} with $\alpha = 2$ we have $K = \cosh( \frac12 \pi \surd 7) /(4 \pi)$ as claimed.   \end{proof}

For the product-moment function $c(s,t) = E\{X(s) X(t)\}$ we have this.
\begin{theorem} \label{thm:9}
For $\tau \leqslant s < t$
\begin{equation} \label{eq:csoln}
c(s,t) = (1+t) \left\{ \frac{q(s)}{1+s} + e ^s\left[(1+s)Q'-(2+s)s q \right] \int_s^t\frac{e^{-y}}{y(1+y)^2} dy\right\},
\end{equation}
and if $s, t \to \infty$, with $s \leqslant t$  then, 
\begin{equation}
t^{-2} c(s,t) \to \begin{cases}\frac{2}{3} \theta K & \text{if $s/t \to \theta \in (0,1)$},\\
                                     K& \text{if $s=t$}.
                       \end{cases}
\end{equation}
\end{theorem}
\begin{proof}
From \eqref{eq:cdot1} we have 
$ t c'' + (1+t) c' - c =0,$ where $c = c(u,t)$, $u <t$ and $c'$ indicates that differentiation is with respect to $t$.
This is essentially \eqref{eq:mddot}, so that we have as before
\begin{equation}\label{eq:generalcsoln} c(s,t) = (1+t) \left[A(s) + B(s) \int_s^t \frac{e^{-y}}{y(1+y)^2}dy\right],
\end{equation}
for suitable $A(s)$ and $B(s)$.  The boundary conditions at $t=s$ are
$$
c(s,t)\big| _{t=s} = q(s) \quad \text{and} \quad \frac{\partial c(s,t)}{\partial t}\big| _{t=s} = Q'(s)/s - q(s),
$$
the latter following from \eqref{eq:cdot1} and \eqref{eq:Qdot}. The required result \eqref{eq:csoln} then follows.

Now set $s= \theta t$ in \eqref{eq:csoln}, where $\theta$ is a fixed number between $0$ and $1$. As $s,t \to \infty$, either integrating by parts or by use of $8.215$ in \citet{GR}, we find that the leading term in the asymptotic expansion of the integral term is $(\theta t)^{-3} e^{-\theta t}$.  
From theorem \ref{thm:8}, we have $q \sim Kt^2$ and hence $\int_0^t q(u) du \sim K t^3/2$, so that from \eqref{eq:qdot} $Q \sim  K t^4/6$  and hence $Q' \sim  2K t^3/3$.
Substituting these asymptotic results in \eqref{eq:csoln}, after some reduction, we have 
$$ \lim_{t \to \infty} t^{-2} c(\theta t,t ) = \textstyle{\frac23} \theta K,$$
as required. Once again we note that this is similar to the behaviour of the product-moment in the discrete case.  
\end{proof}
For the third moment $E\{X^3(t)\}$, we remark that a similar asymptotic analysis may be pursued. Introducing the notation
$S_j(t) = E\{\int_0^t X^j(u)du\}$ and 
$$\alpha_j = E\{X^j(t) [S_1(t)]^{3-j}\},\quad \beta_j = E\{S_j(t) [S_1(t)]^{3-j}
\},\quad\gamma_j = E\{X^j(t)S_{3-j}(t)\},$$
and using the usual conditional expectation arguments, we have
\begin{eqnarray*}
\alpha_0'&=& 3 \alpha_1, 
\quad \alpha_1' = -\alpha_1 +2 \alpha_2 +\frac{2}{t}\alpha_0,
\quad \alpha_2' =  -\alpha_2 + \alpha_3 + \frac{2}{t} \beta_2 + \frac{2}{t^2} \alpha_0,
\\
\quad \alpha_3' &=& -\alpha_3 +\frac{2}{t} \beta_3 +\frac{6}{t}\beta_2, 
\quad \beta_2'= \alpha_2 + \gamma_1 ,
\quad \beta_3' = \alpha_3, 
\quad \gamma_1' = -\gamma_1 + \alpha_3 + \frac{2}{t} \beta_2. 
\end{eqnarray*}
Reducing this system to a single differential equation for $\beta_3 = E\{\int_0^t X^3(u)du\}$ yields
\begin{equation}
\begin{split}
t^4 \beta_3^\text{(vii)}+4(t+4)t^3 \beta_3^\text{(vi)}
+2(3 t^2+21 t+37) t^2 \beta_3^\text{(v)}\\
 +2(2t^3+15t^2+40t+54)t \beta_3^\text{(iv)}
+(t^4-2t^3-44t^2-80t+36)\beta_3'''\\
-(6t^3+32t^2+72t+112)\beta_3''
+(18t^2+92t+136)\beta_3' -12(2t+3)\beta_3 = 0.
\end{split}
\end{equation}
Again following \citet{Erdelyi} we consider the asymptotic expansion of $\beta_3$ developed in inverse powers of $t$ for large $t$, and find the leading  term by substituting trial solutions of the form $\beta_3 = t^\sigma e^{\delta t}$ and $t^\rho.$  These show that $\delta$ is either $0$ or $-1$ and $\rho $ is either $2$, $3$ or $4$.  We conclude that $E\{\int_0^t X^3(u)du\}$ grows as $t^4$, and hence that $E\{X^3(t)\}$ grows as $t^3$.

As with the adding and $p$-adding processes a martingale is available:  
\begin{lemma}
Let
$$M(t) = \frac{A(t)}{t(1+t)} S(t) + \frac{B(t)}{t(2+t)}  X(t), \quad t \geqslant \tau,$$
where $S(t) = \int_0^t X(v)dv$ and where  $A(t) = \int_t^\infty a(t)/a(v) dv $ with $a(v) = e^{-u} v(1+v)$ and $B(t) \int_t^\infty b(t)/b(v) dv $ with $b(v) = e^{-v} v^2(2+v)^2$.  Then $(M(t); t\geqslant \tau)$ is a martingale with respect to the increasing $\sigma$-fields $(\mathcal{F}(t);t\geqslant \tau)$ generated by $(X(t))$ or equivalently $(S(t))$. Furthermore, there exists a non-degenerate random variable $M$ with finite variance, such that $M(t)$ converges to $M$ almost surely and in mean-square as $t \to \infty$, where 
$$E M = \frac{A(\tau)}{\tau(1+\tau)}\int_{0}^\tau x(v)dv + \frac{B(\tau) x(\tau)}{\tau(2+\tau)}.$$  
Note that both $A(t)$ and $B(t)$ converge to $1$ as $t \to \infty$.
\end{lemma}
\begin{proof}
From theorem \ref{meanDEsoln} we know that 
$E X(t) | \mathcal{F}(u)$ and hence $E S(t) |\mathcal{F}(u)$ depend only on $X(u)$ and $S(u)$ for $\tau \leqslant u \leqslant t$. 

Now consider the random process $(M(t) = \alpha(t)S(t) + \beta(t)X(t), t \geqslant \tau)$ where $\alpha$ and $\beta$ are differentiable functions. Let $m(t) = E X(t) | \mathcal{F}(u)$ and $\theta(t) = E S(t) |\mathcal{F}(u)$. The conditional expectation $E M(t)| \mathcal{F}(u)$ is then $\alpha(t)\theta(t) + \beta(t)m(t))$.

For $(M(t))$ to be a martingale we require that $$E\, \big(\alpha(t) S(t)  + \beta(t) X(t)\big)|\mathcal{F}(u) = \alpha(u)S(u) + \beta(u)X(u), \tau \leqslant u \leqslant t$$ 
in particular the left-hand side should not depend on $t$, and consequently
\begin{align*}
\frac{d}{dt} \big(\alpha \theta + \beta m \big) &= \alpha'\theta + \alpha\theta' + \beta' m + \beta m'\\
&=\alpha'\theta + \alpha m + \beta' m + \beta \left(\frac{2}{t}\theta-m\right)\\
&=\theta\left(\alpha' +\frac{2}{t}\beta\right) + m(\alpha + \beta' -\beta) =0
\end{align*} 
Solving the pair of differential equations $(\alpha' +\frac{2}{t}\beta =0; \alpha + \beta' -\beta=0)$, we have 
$$\alpha = (1+t)e^{t}\int_t^\infty \frac{e^{-v}}{v(1+v)^2} dv = \frac{1}{t(1+t)}\int_t^\infty\frac{e^{t-v}t(1+t)^2}{v(1+v)^2} dv = \frac{A(t)}{t(1+t)}$$
and 
$$\beta = t(2+t)e^{t}\int_t^\infty \frac{e^{-v}}{v^2(2+v)^2} dv = \frac{1}{t(2+t)}\int_t^\infty\frac{e^{t-v}t^2(2+t)^2}{v^2(2+v)^2} dv = \frac{B(t)}{t(2+t)}.$$
Both of the integrals $A(t)$ and $B(t)$ converge to 1 as $t \to \infty$ by dominated convergence. 

Finally from theorem \ref{thm:9}, we see that $E M^2(t)$ is uniformly bounded and so the probabilistic limit follows from the martingale limit theorem.  
\end{proof}
Recalling lemma \ref{thm:p-martingale} and its corollary at the end of section \ref{sec:3}, it is seen, by exactly the same argument, that $S(t)/t^2$ converges to $M$ a.s. and in  m.s. as $t \to \infty$.   And this implies very similar conclusions for the behaviour of the sample paths of the continuous-time process as that given in section \ref{sec:2} for Ulam's discrete-time process.  Once again, we see the great utility of a suitable martingale.
\section{Generalized random adding}

A natural generalization of the simple adding process is to allow weighting and non-uniform selection from the past.  We define such a process thus:

\medskip
{\sc Definition 5.1}. With the notation and structure of definition 4.1, at jump times $(T_n)$, 
set 
$$X(T_n) = A X(U_n) + B X(V_n), \quad n \geqslant 1,
$$ 
where now $(U_r)$ and $(V_r)$ comprise sequences of independent random variables with respective distribution functions
\begin{eqnarray*} 
P \{U_r \leqslant u| T_r =t\}  &=& (u/t)^{\alpha}, \quad 0 < u < t, \\ 
P \{V_r \leqslant v| T_r =t\}  &=& (v/t)^{\beta}, \quad 0 < v < t.
\end{eqnarray*}
Here $A$ and $B$ are non-zero constants, $\alpha$ and $\beta$ are positive constants.  As in section \ref{sec:simplecont}, we assume that the initial values in the process are fixed at $X(t) = x(t)$ for $0\leqslant t \leqslant \tau$.  

\begin{theorem} \label{thm:10}
Let $m(t) = E X(t) $ then $m(t)$ grows asymptotically as $t^\sigma$ as $t \to \infty$, 
where $\sigma$ is a root of the following equation; in general that root having the larger real part:  
\begin{equation} \label{eq:sigma}
\sigma^2+[\beta(1-B) + \alpha(1-A)]\sigma +[(1-A-B)\alpha\beta] =0.
\end{equation}
\end{theorem}
\begin{proof}
Conditioning on the events of the Poisson process $(N(t))$, we have
$$
m' + m = \frac{A\alpha}{t^{\alpha}} \int_0^t u^{\alpha-1} m(u) du + \frac{B\beta}{t^{\beta}} \int_0^t v^{\beta-1} m(v) dv.
$$
Differentiating with respect to $t$, we obtain, for $t \geqslant \tau$,
\begin{equation}\label{eq:genmDE}
\begin{split}
(1-A-B)\alpha\beta m+ t^2 m''' + [(\alpha+ \beta +1)t + t^2]m''\\ + \{\alpha\beta + [1+\beta(1-B) +\alpha(1-A)]t\} m'=0.
\end{split}
\end{equation} 
Following \citet{Erdelyi}, substituting the usual trial solutions in \eqref{eq:genmDE} and equating coefficients of the highest order terms, we find that  $m(t)$ grows asymptotically as $t^\sigma$, where $\sigma$ is given by \eqref{eq:sigma}. Note that when $\alpha=\beta=1$ and $A=B=1$ then $\sigma=1$, as we know from \ref{meanDEspecial}.
\end{proof}   

We investigate the implications of equation \eqref{eq:sigma} beginning with the question of when the roots are imaginary, corresponding to potentially oscillatory behaviour for $m(t)$.  For brevity of notation, we write $x= 1-A$ and $y=1-B$.  The roots $\sigma_1$ and $\sigma_2$ of \eqref{eq:sigma} are real or imaginary according as the function
$$f(\alpha,\beta,x,y) = \alpha^2x^2 + 2 \alpha\beta x y + \beta^2 y^2 - 4\alpha\beta(x+y-1)
$$
is greater than or equal to, or less than, zero.

We observe that $f=0$ defines a parabola $\mathcal{P}$ in the $(x,y)$ plane for suitable fixed $\alpha$ and $\beta$. Writing $f$ as
$$
f(\alpha,\beta,x,y)= \left[\alpha x + \beta y - \frac{2 \alpha \beta (\alpha + \beta)}{\alpha^2 + \beta^2}\right]^2 - \frac{4 \alpha \beta(\beta-\alpha)}{\alpha^2 + \beta^2}\left[\beta x - \alpha y+ \frac{\alpha^3-\beta^3}{\alpha^2+\beta^2}\right],
$$
we see that, the axis of $\mathcal{P}$ is 
$$\alpha x + \beta y = \frac{2 \alpha \beta(\alpha + \beta )}{\alpha^2+\beta^2},
$$
and the tangent $T$ at the vertex is 
$$\beta x - \alpha y + \frac{\alpha^3 - \beta^3}{\alpha^2 + \beta^2} = 0.
$$
Note that the roots $\sigma_1$ and $\sigma_2$ are complex inside the parabola (with an obvious convention).  If $\alpha > \beta$, then the parabola is above $T$; if $\alpha< \beta$, then $\mathcal{P}$ lies below $T$; if $\alpha= \beta$ then the case is degenerate and $\mathcal{P}$ is the line $x+y=2$ (corresponding to $A+B=0$).  

Now let us consider the points $(1,0)$ and $(0,1)$ with respect to $\mathcal{P}$.  The polar of $(1,0)$, i.e.\ the chord of contact of the tangents from the point $(0,1)$, is
$$ f_1 = \alpha(\alpha x+ \beta y) - 2 \alpha \beta(x+y-1),$$
and the power of $(1,0)$ with respect to $\mathcal{P}$ is $f_{11} = \alpha^2.$  Therefore the tangents to $\mathcal{P}$ meeting at $(1,0)$ are given by the line pair
$$
0 = f_1^2 - f f_{11} = \alpha^2 (\beta- \alpha) (x-1)(x+y-1).
$$
Likewise, the tangents to $\mathcal{P}$ from $(0,1)$ are the line pair $(y-1)(x+y-1) = 0$.  
From an early result attributed to \citet{Lambert} we know that the locus of the focus of parabolas with three specified tangent lines is the circle though the vertices of the triangle formed by the intersections of the lines; in this case the points $(0,1)$, $(1,0)$ and $(1,1)$.  
\begin{figure}[h]
	\centering
		\includegraphics[width=0.7\textwidth,trim=2cm 7cm 1cm 7cm]{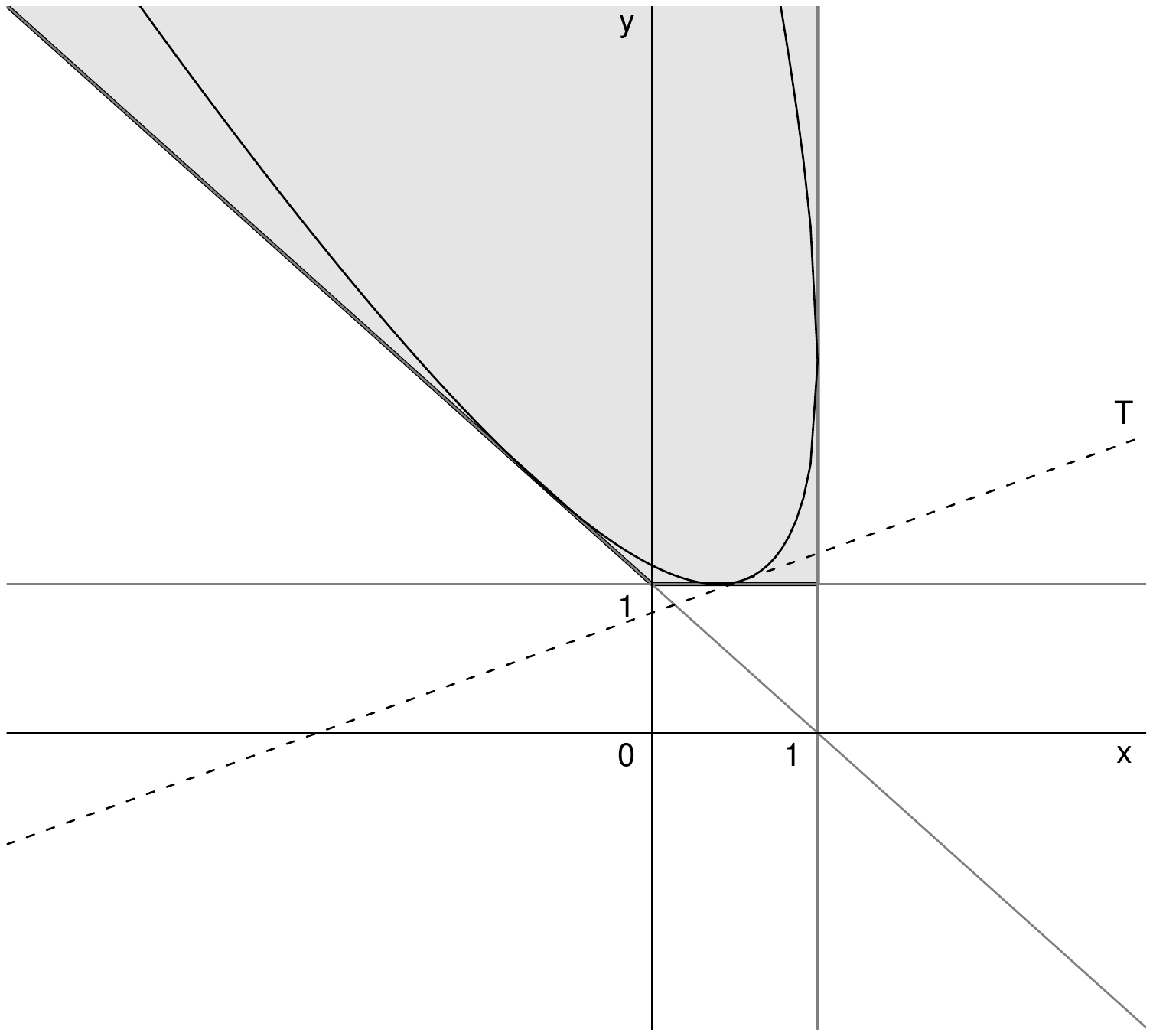}
	\caption{Region of oscillatory behaviour (shaded) in the case $\alpha > \beta$. When $\alpha < \beta$ the oscillatory region is given bv by reflection in the line $y=x$. \label{fig:4}  }
\end{figure}
As $\alpha$ and $\beta$ run over all positive values, the three points of contact with $x=1$, $y=1$, and $x+y=1$, are seen to trace all points of these lines except those that lie in the region $\{x < 1\} \cap \{y <1\}$.  Thus these lines delineate the envelope of the parabolic region in which $m(t)$ is oscillatory; see figure \ref{fig:4}.

Secondly, we consider whether $\sigma_1$ or $\sigma_2$ has positive real part (corresponding to potentially unbounded solutions for $m(t)$).  If the roots are imaginary, $(\alpha,\beta)$ lying inside $\mathcal{P}$, then (being conjugate) they have a positive real part if $\alpha x + \beta y < 0.$
If the roots are real, then at least one is positive if either $\alpha x + \beta y <0$, or $x+y<1$.  The nature of the asymptotic behaviour of $m(t)$ as $t \to \infty$ is thus given in terms of the parameters $x$ and $y$;  see figure \ref{fig:5}.
\begin{figure}[h]
	\centering
		\includegraphics[width=0.7\textwidth,trim=2cm 7cm 1cm 7cm]{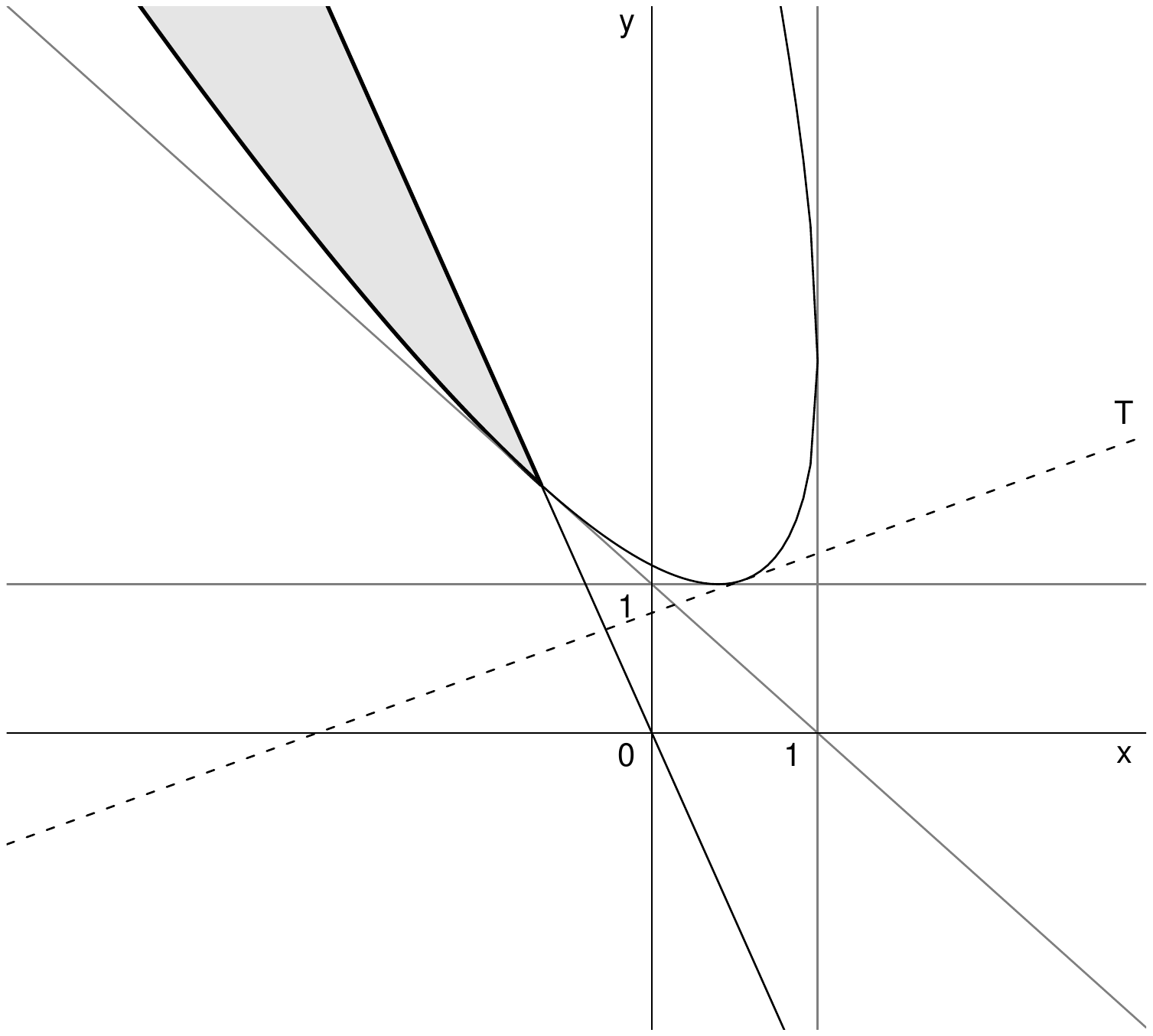}
	\caption{Region in which the roots of \eqref{eq:sigma} are complex with positive real part (shaded).   \label{fig:5}  }
\end{figure}
Alternatively, we may regard $A$ and $B$ and hence $x$ and $y$ fixed, and consider the quadratic form in $\alpha$ and $\beta$ given by 
$$\phi = \alpha^2 x^2 + 2 \alpha \beta (xy-2x -2y +2) +\beta^2 y^2.
$$
A necessary condition for this to take negative values is that it should be a real line pair, for which a necessary and condition is that $(xy-2 x - 2 y +2)^2 > x^2 y^2$ which is equivalent to $(x-1)(y-1)(x+y-1) < 0$.  Note that the two regions of oscillatory behaviour in figure \ref{fig:5} do indeed satisfy this constraint.  The oscillatory region in the $(\alpha,\beta)$ plane then comprises those opposite angles lying between the line pairs in which $\phi$ is negative.  In the case when $x<1$, $y<1$ and $x+y<1$, no part of this region lies in $\{\alpha >0\} \cap\{\beta >0\}$, so there are no oscillatory solutions there.  

Of course, we may also seek a solution of \eqref{eq:genmDE} as a power series in $t$.  In the usual way, the indicial equation is found to be $c(c-\alpha+1)(c-\beta+1)= 0 $. which supplies the required three linearly independent solutions in the ordinary case when $\alpha$ and $\beta$ are neither equal nor differ by an integer.  In these cases the method of Frobenius may generally be employed to yield the required distinct solution in series.  We refrain from an extended discussion.  However we do mention the special case when the boundary condition is $\tau=0$ with $x(0)=1$. In this instance, in general, the power series corresponding to $c=0$, with the form $m(t) = 1 + \sum_{r=1}^\infty a_r t^r $,  supplies the solution that is regular at the origin.  For example, if $A=B=1$ then it is seen that
$m(t) \sim t^\sigma$ with $\sigma = [\alpha\beta]^{1/2}.$  If it is further assumed that $\alpha\beta = 4$ where neither $\alpha$ nor $\beta$ is an integer, then \eqref{eq:genmDE} has the solution $m(t) = 1+ t + 3 t^2/[2(\alpha+\beta) +12],$ by inspection.  By the remarks above, this is the required $m(t)$ satisfying the boundary conditions $m(0) = m'(0) =1$ and is such that $m(t)$ grows quadratically as $t \to \infty$.
\subsection{The second moment in the generalized case}
In considering the second moment $q(t) = E X^2(t)$ of the process $X(t)$ of definition $5.1$, we will make the assumption that $\alpha=\beta > 0$, thus excluding the oscillatory behaviour.  We have this:
\begin{theorem} \label{thm:11} 
As $t \to \infty$, $q(t) \sim K t^\sigma$ where 
$$ \sigma = \alpha \max \{A^2 +B^2 -1, 2(A+B-1)\},$$
and $K$ is a constant depending on $A$, $B$ and $\alpha$ and initial conditions.
\end{theorem}
\begin{proof}
Let $C_1=A+B$ and $C_2= A^2+B^2$, then by conditioning on the events of the Poisson process during $(t,t+h)$, we find in the usual way that
\begin{equation}\label{eq:genQdot}
q'+q = \frac{\alpha C_2}{t^{\alpha}} \int_0^t u^{\alpha-1} q(u) du + \frac{2 A B \alpha^2 Q}{t^{2\alpha}},
\end{equation}
where $Q = \int_0^t\int_0^t (uv)^{\alpha-1} c(u,v) du dv$ and $c(u,v) = E\{ X(u) X(v)\}$. Likewise 
$$ \frac{\partial c(u,t)}{\partial t} + c(u,t) = \frac{\alpha C_1}{t^{\alpha}} \int_0^t y ^{\alpha-1} c(u,y) dy, \quad u<t,
$$
with a similar equation for $\partial c(t,v)/\partial t$, when $v<t$.  Differentiating $Q$ we find
$$ \frac{d}{dt} \left(t^{-\alpha+1} \frac{dQ}{dt}\right) + t^{-\alpha+1}\frac{dQ}{dt} - 2t^{\alpha-1} q = 2\alpha C_1 t^{-\alpha} Q,
$$ where we have substituted for $\partial c(u,t) /\partial t$ and $\partial c(t,v) /\partial t$, as necessary.  Substituting for $Q$ throughout, using \eqref{eq:genQdot}, we have this equation for $q(t)$:
\begin{equation}
\begin{split}
t^3 q^{(\text{iv})} +2 (t+2\alpha +1) t^2q'''
+\{t^3+[(7-2 C_1-C_2)\alpha+3] t^2+(5\alpha^2+\alpha)t\}q''\\
+ \{[(3-2 C_1-C_2)\alpha+1] t^2+[(7-4 A B-2 C_1-3 C_2)\alpha^2+\alpha]t + 2\alpha^3 -2\alpha^2\}q'\\
+ [2(C_2-1)(C_1-1)t\alpha^2-2\alpha^2(\alpha-1)(2AB+C_2-1)]q = 0.
\end{split}
\end{equation}
\begin{figure}[ht]
	\centering
		\includegraphics[width=0.7\textwidth,trim=2cm 7cm 1cm 7cm]{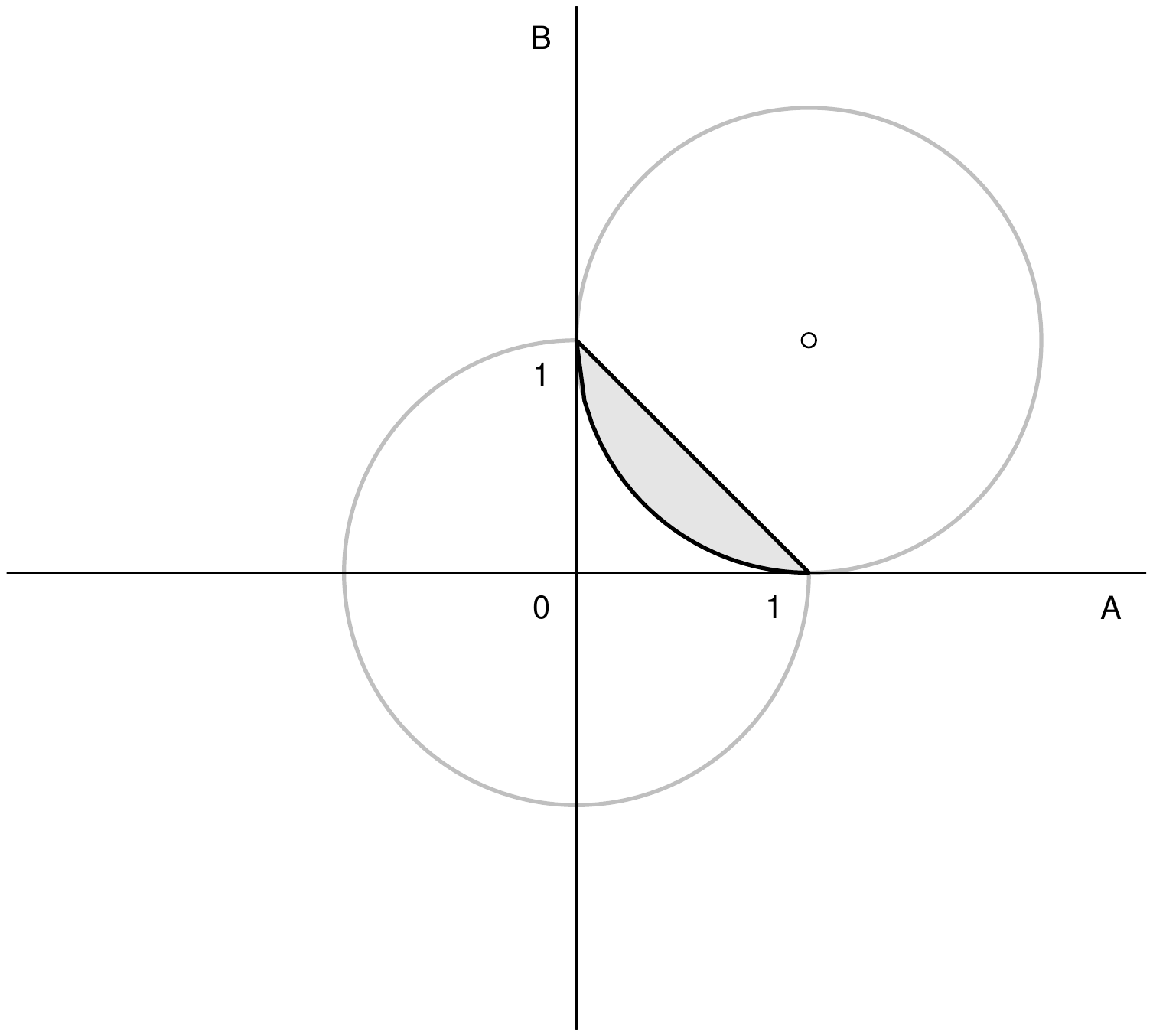}
	\caption{Outside the circle centred at $(1,1)$, the second moment increases or decreases as $t^{\sigma_1}$ depending on the sign of $\sigma_1 = \alpha(A^2+B^2 -1)$   Within this circle it grows as $t^{\sigma_2}$, where $\sigma_2 = 2\alpha(A+B-1)$; decreasing in the shaded region and otherwise increasing.  The Ulam case is at the point $\scriptstyle{\circ}$. \label{fig:6}  }
\end{figure}
Again following \citet{Erdelyi}, we find $q(t)$ grows as $t^\sigma$ for large $t$, where $\sigma$ is given by
$$\sigma(\sigma-1) + [(\alpha+1) + \alpha(2-C_2-2C_1)]\sigma + 2\alpha^2(1-C_1)(1-C_2) = 0.
$$
This factorises into
$$[\sigma + 2\alpha(1-C_1)][\sigma + \alpha(1-C_2)] =0,
$$
giving the two roots as claimed. The leading term is therefore $t^\sigma$ with $\sigma = \alpha(A^2+B^2 -1)$ everywhere in the $(A,B)$ plane, except inside the circle $(A-1)^2 + (B-1)^2=1$.  This is illustrated in figure \ref{fig:6}. Numerical solutions of the differential equations, for various initial conditions and parameter values, demonstrate exactly the behaviour described theorems \ref{thm:10} and \ref{thm:11} 
\end{proof}
Finally, we briefly discuss the effects on $(X(t))$ if at each jump $X(T_r) = A_r X(U_r) + B_r X(V_r)$ where now $(A_r)$ and $(B_r)$ comprise sequences of independent random variables, also independent of $(U_r, V_r)$, with means $\mu_A= E A_r$ and $\mu_B = E B_r$ respectively.  It is easy to see that in \eqref{eq:sigma} and \eqref{eq:genmDE}, one simply replaces $A$ and $B$ by $\mu_A$ and $\mu_B$.  The essential conclusions in figure \ref{fig:6} are the same, with some relabelling.  For the second moment, we note that the product moment $E (AB)$ is irrelevant to first order.  The end result is that $q(t)$ grows with $t^\sigma$ where now $\sigma= \alpha\max\{2(\mu_A+\mu_B-1), E A^2 +  E B^2-1\}$.   The nature of the final figure will then be similar, but dependent on the actual distributions of $A$ and $B$, as expressed in their first two moments.

\subsection{Generalized adding processes in discrete time} Of course, one can also define such generalized adding processes in discrete time, but we avoid discussing these in detail.  We strongly conjecture that they will show essentially the same behaviour as continuous-time generalized processes, and we sketch one example to illustrate this. In the usual way, in the notation of \eqref{eq:BSU} and lemma \ref{thm:martingale}, let
$X_{n+1} = X_{U(n)} + 2 X_{V (n)}$, $n \geqslant 2$;  and define  $M_n = S_n/(n(n + 1)(n + 2))$ , where $S_n = \sum_{k=1}^n X_k$, as usual.
Then $(M_n; n \geqslant r)$ is a martingale with respect to the increasing sequence of $\sigma$-fields $(F_n; n \geqslant r)$ generated by the sequence $(X_n)$, or equivalently by $(S_n)$.

The convergence of this martingale, which we refrain from proving, shows that $E X_n$ grows asymptotically like $n^2$ ; and we note that this is entirely consistent with the result \eqref{eq:sigma} of theorem \ref{thm:10}, in the case when $A=1, B=2$, and $\alpha = \beta = 1$.   Clearly numerous other similar martingales can be recruited to consider the behaviour of $X_n$ and $S_n$ in more general cases. 

As an illustration, using the same notation, with the recurrence $X_{n+1} = A X_{U(n)} + B X_{V(n)}$, $n \geqslant 2;$ for suitable constants $A$ and $B$, we find that 
$M_n = [(n - 1)!/(n + A + B - 1)!]\,S_n$ 
satisfies the martingale condition wrt $(F_n)$; with the usual falling factorial convention for $(n+c)!$, for $n+c$ not an integer.   And then $M_n$ is a martingale for those values of $A$ and $B$ such that $E |M_n|$ is finite.  

The convergence of this martingale, whose proof we omit, entails the convergence of $S_n/n^{A+B}$  to some r.v. as $n \to \infty$, using the
Stirling-DeMoivre formula for large $n$. From which one may deduce that $E X_n$ grows like $n^{A+B-1}$, in agreement with the continuous time results.

The determination of the implicit constraints on $A$ and $B$, analogous to those given by theorems \ref{thm:10} and \ref{thm:11}, is an open problem. 

\section{Conclusion}
We have considered Ulam's random adding process, introduced in \citet{BSU}, and verified the authors' conjecture about the quadratic growth of the process's second moment.  We have also introduced a number of new, more general random adding processes, in both discrete and continuous time, showing that their moments exhibit similar behaviour.   Furthermore, for the basic simple Ulam process of section \ref{sec:2}, we showed that $M_n = S_n/(n(n+1))$ converges almost surely and in mean-square to a limiting random variable $M$. The result depended crucially on the identification of a martingale. Related martingales were also identified for the $p$-adding and continuous-time processes, which established the a.s. and m.s. convergence of $S_n/n^2$ and $S(t)/t^2$ respectively, leading to similar conclusions about the behaviour of their sample paths. We have been unable to find suitable martingales for the generalized random adding processes of section 5, though it seems likely that similar convergence results will apply. A possible approach is to establish mean-square convergence by showing that the random sequence is Cauchy in mean-square. Our preliminary investigations suggest that limit results of the types given in theorems \ref{thm:5} and \ref{thm:9} are not precise enough for this purpose and that higher order approximations will be necessary. Finally, we note that there are many further obvious and interesting open problems about almost every aspect of this largely unexplored family of random processes.

\paragraph{Remark:} The result \eqref{eq:qlimit} in the special case $r=1=x_r$, was obtained but not published by \citet{Turner}, while working with Mark Kac who analysed another of Ulam's history dependent recurrences \citep{Kac}.
\newpage
\bibliographystyle{apalike}

\vspace{5ex}
\hrule
\vspace{1ex}
\urlstyle{sf}
Email address: {\sf peter.clifford@jesus.ox.ac.uk}\\
URL: \url{https://www.stats.ox.ac.uk/~clifford}

Email address: {\sf david.stirzaker@sjc.ox.ac.uk}\\
URL: \url{https://www.sjc.ox.ac.uk/discover/people/professor-david-stirzaker}

\vspace{2ex}
\hrule

\end{document}